\documentclass[11pt,reqno,oneside]{amsart}

\usepackage{graphicx}
\usepackage{amssymb}
\usepackage{epstopdf}

\usepackage[a4paper, total={6in, 9.1in}]{geometry}

\usepackage{amsmath,amsfonts,amsthm,mathrsfs,amssymb,cite}
\usepackage[usenames]{color}

\usepackage{subfigure} 
\usepackage{framed}

\usepackage{enumerate}
\usepackage{hyperref}
\usepackage{xcolor}
\hypersetup{
  colorlinks,
  linkcolor={blue!80!black},
  urlcolor={blue!80!black},
  citecolor={blue!80!black}
}
\usepackage[capitalize,noabbrev]{cleveref}

\usepackage{booktabs} 
\usepackage{array} 
\usepackage{paralist} 
\usepackage{verbatim} 
\usepackage{tabularx}
\usepackage{amsmath,amsfonts,amsthm,mathrsfs,amssymb,cite}
\usepackage[usenames]{color}
\usepackage{bm}
\usepackage[T1]{fontenc}

\newtheorem{thm}{Theorem}[section]
\newtheorem{cor}[thm]{Corollary}
\newtheorem{lem}{Lemma}[section]
\newtheorem{prop}{Proposition}[section]
\theoremstyle{definition}

\theoremstyle{remark}

\newtheorem{rem}{Remark}[section]
\numberwithin{equation}{section}

\allowdisplaybreaks

\title[stability estimates for Master equation]{Lipschitz Stability Estimates for Master Fields with Measure-Variable Regularity in Mean Field Games}

\author{Chen Geng}
\address{School of Mathematics, Harbin Institute of Technology, Harbin, China.
Department of Mathematics, City University of Hong Kong, Hong Kong,
China}
\email{gengchen@stu.hit.edu.cn, chengeng3-c@my.cityu.edu.hk}

\author{Hongyu Liu}
\address{Department of Mathematics, City University of Hong Kong, Hong Kong,
	China}
\email{hongyliu@cityu.edu.hk, hongyu.liuip@gmail.com}

\author{Minghui Song}
\address{School of Mathematics, Harbin Institute of Technology, Harbin, China.}
\email{songmh@hit.edu.cn}

\address{}
\email{}

\begin{document}
\maketitle

\begin{abstract}
	
In this paper, we investigate the Lipschitz stability of the master field and its functional derivative with respect to the measure variable for the master equation in mean field games. A key novelty of this work is the derivation of a nonlocal parabolic partial differential equation governing the functional measure derivative along a characteristic measure flow. Based on this equation, we establish a Carleman estimate for the difference of the measure derivatives associated with two sufficiently regular classical solutions. We also establish a Carleman estimate for the difference of the corresponding master fields, yielding a terminal-to-interior Lipschitz stability estimate along a characteristic measure flow. Together with stability of the associated measure flows, these estimates are used to control the resulting source terms. We then obtain a terminal-to-interior Lipschitz stability estimate for both the master field and its functional measure derivative in terms of the discrepancies of the terminal data. The result provides quantitative control of the master field with measure-variable regularity, and develops a Carleman-based framework for the stability analysis of master fields and their measure derivatives in mean field game master equations.

	\medskip


\end{abstract}

\section{Introduction}
\subsection{Background and motivation.}

Mean field games provide a mathematical framework for describing the strategic
interactions of a large population of rational agents. This theory was introduced
independently by Lasry and Lions \cite{LasryLions2006a,LasryLions2006b,LasryLions2007}
and by Huang, Caines and Malham\'e
\cite{HuangMalhameCaines2006,HuangCainesMalhame2007}.
It provides an asymptotic description of Nash equilibria in stochastic differential games with a large number of weakly interacting agents, where the interaction among individual players is mediated through the distribution of the population.
In the classical PDE formulation, the equilibrium is characterized by a coupled forward--backward system:
a Hamilton--Jacobi equation for the value function of a representative agent and a
Fokker--Planck equation for the evolution of the population distribution. The coupling reflects the equilibrium mechanism: each agent optimizes according to the population distribution,
and the population evolves under the optimal feedback generated by the value function.

The master equation is a fundamental object in mean field game theory. Unlike the
mean field game system associated with a fixed initial distribution, the master equation describes
the dependence of the value function on the population distribution. In this paper, we consider the solution of the master equation
\[
U=U(t,x,m), \qquad (t,x,m)\in [0,T]\times \mathbb T^d\times \mathcal P(\mathbb T^d),
\]
where \(\mathcal P(\mathbb T^d)\) denotes the space of probability measures on the
torus $\mathbb{T}^d$. 
By treating the population distribution as an additional variable, the master equation provides a unified description of the family of MFG systems associated with different initial distributions. 
It also plays a central role in the rigorous analysis of the mean field limit of \(N\)-player Nash systems, since the limiting value function is naturally expected to depend on both the state of a representative player and the empirical distribution of the remaining players \cite{CardaliaguetDelarueLasryLions2019}. 
Owing to its dependence on the measure variable and the resulting nonlocal interactions, the master equation is an infinite-dimensional, nonlinear and nonlocal PDE on the space of probability measures; see also
\cite{BensoussanFrehseYam2013,CarmonaDelarue2018,CardaliaguetCirantPorretta2023}
for related developments.

The dependence of $U$ on the probability measure $m$ requires a differential calculus
on the space of probability measures. In the analysis of the master equation, two related
notions of measure derivatives are commonly used. The functional derivative
$\delta U/\delta m$ describes the first-order variation of the master field with respect to
perturbations of the population distribution, while the intrinsic derivative $D_mU$,
also referred to as the Lions derivative, describes the corresponding variation under
infinitesimal transport of the population mass. These derivatives play a fundamental role
in the formulation and regularity theory of master equations and in the analysis of their dynamics. In particular, the functional measure derivative provides a natural
description of the first-order sensitivity of the master field with respect to the population
distribution. We refer to
\cite{AmbrosioGigliSavare2008,Villani2009,LionsCollege,
	CardaliaguetNotes2013,CardaliaguetDelarueLasryLions2019,
	CardaliaguetCirantPorretta2023}
for the calculus on spaces of probability measures and its applications to master equations.

The functional derivative with respect to the measure variable also carries a clear
modeling significance, since it quantifies the sensitivity of an individual value to
variations in the distribution of the population. A representative example is provided
by mean field models of traffic congestion, where the population measure describes the
distribution of vehicles over a transportation network and the value function represents
the optimal travel cost of an individual driver; see, for instance,
\cite{HuangChenDiDu2021}. In the corresponding master-equation formulation,
$\delta U/\delta m(t,x,m,y)$ describes how a perturbation of the traffic distribution
around the location $y$ influences the value of a driver located at $x$.
Consequently, knowledge of the measure derivative over the whole time interval provides
information on how congestion effects are transmitted through the transportation network.
For example, it may help identify regions whose traffic perturbations have a significant
influence on other parts of the network, as well as regions in which drivers are
particularly sensitive to congestion occurring elsewhere. Such sensitivity information
is relevant to the analysis of traffic bottlenecks, route planning, and the design and
optimization of transportation networks. Similar interpretations arise in other
large-population models. In systemic-risk models, the population distribution describes
the financial states of a large number of interacting banks
\cite{CarmonaFouqueSun2015}, while in dynamic energy-demand models it represents the
distribution of consumption states of a large population of users
\cite{Bauso2017}. In these applications, the measure derivative characterizes how
changes in the collective configuration of the system affect the value of an individual
agent. 
Therefore, obtaining the master field and, in particular, its functional measure derivative throughout the time interval from terminal-time information, together with corresponding stability estimates, provides meaningful information on the evolution of the value function and its population sensitivity and constitutes a natural problem in the analysis of master equations.

The theory of master equations has developed substantially in recent years, with
particular emphasis on well-posedness and regularity, especially with respect to the
measure variable. Under suitable monotonicity and regularity assumptions, Cardaliaguet,
Delarue, Lasry and Lions \cite{CardaliaguetDelarueLasryLions2019} established a systematic
classical theory for the master equation and its connection with finite-player Nash systems.
Probabilistic approaches to classical solutions were further developed in
\cite{CarmonaDelarue2018,ChassagneuxCrisanDelarue2022}.
Beyond the global monotone framework, short-time solvability of first-order master
equations was established by Gangbo and \'{S}wi\k{e}ch
\cite{GangboSwiech2015} and later further developed in
\cite{Mayorga2020}, while Cardaliaguet, Cirant and Porretta
\cite{CardaliaguetCirantPorretta2023} developed a splitting approach and obtained
systematic estimates for spatial and measure derivatives of the master field.
Related extensions include master equations in bounded domains
\cite{Ricciardi2022}.
Alternative global-in-time theories based on displacement convexity and displacement
monotonicity, together with generalized notions of monotone solutions for master equations, were developed in \cite{GangboMeszaros2022,GangboMeszarosMouZhang2022,
	CardaliaguetSouganidis2022,Bertucci2023}.
More recently, Mou and Zhang
\cite{MouZhang2024,MouZhang2025} developed well-posedness theories under weaker
regularity and alternative structural conditions, while Bertucci, Lasry and Lions
\cite{BertucciLasryLions2024} established a theory of Lipschitz solutions beyond the
classical monotonicity framework.
Thus, although several alternative frameworks have been developed, the classical well-posedness theory of master equations is still largely based either on monotonicity-type structural conditions for global-in-time results or on local-in-time arguments under a short-time restriction.
Quantitative stability estimates for master fields with respect to terminal data remain relatively limited. 
In particular, Mou, Zhang and Zhou \cite{MouZhangZhou2025} established a Lipschitz-type stability estimate for classical master fields under perturbations of the Hamiltonian and terminal data through probabilistic approach.
While, a Carleman framework has not
previously been developed for the quantitative stability analysis of the master
equation itself. Complementing these approaches, we establish a
terminal-to-interior Lipschitz stability estimate for the master field along
characteristic measure flows. We further extend this framework to the functional
measure derivative, leading to quantitative stability at the level of the
first-order dependence of the master field on the population distribution.

A closely related line of research concerns quantitative determination, stability,
and inverse analysis for mean field game systems. In this direction, Liu and
collaborators have developed a series of results on the recovery of running and
terminal costs, Hamiltonians, states, stationary states, and state-space anomalies
from various types of observations
\cite{LiuMouZhang2023,LiuZhang2025,LiuYamamoto2023,
	LiuLo2025,LiuLoZhang2025,DingLiuZheng2025,
	LiuLoStationary2025}.
Of particular relevance to the present work, Carleman estimates have been
	successfully applied to quantitative stability and uniqueness for MFG systems,
	yielding H\"older and Lipschitz stability, unique continuation, state determination,
	inverse-source results, lateral Cauchy estimates, and stability under terminal or
	overdetermined observations
\cite{KlibanovAverboukh2024,KlibanovLiLiu2023,
	ImanuvilovLiuYamamoto2023,ImanuvilovLiuYamamoto2024,
	KlibanovLiLiu2024,KlibanovLiLiu2026}.
More recently, inverse analysis has begun to move beyond finite-dimensional MFG
systems toward PDEs posed directly on spaces of probability measures. In particular,
Liu, Qian and Zhang \cite{LiuQianZhang2025} developed an inverse framework for
infinite-dimensional transport equations on Wasserstein space, with the MFG master
equation as one of the principal examples.
These developments provide the methodological background for the present
analysis. We extend the Carleman approach from MFG systems to the master-equation
level and first obtain terminal-to-interior Lipschitz stability for the master field
along characteristic measure flows. We then derive a nonlocal parabolic equation for the functional measure derivative and establish its corresponding terminal-to-interior Lipschitz stability. 
To the best of our knowledge, quantitative stability of the
functional measure derivative with respect to perturbations of the terminal data has
not previously been established. Thus, our results provide a Carleman framework for
stability both at the level of the master field and at the level of its first-order
sensitivity with respect to the population distribution.

This work is concerned with the quantitative stability of the master field and its functional measure derivative. Building on the differentiability theory of master equations and the Carleman-based stability analysis developed for mean field game systems, we study the propagation of terminal information first at the level of the master field and then at the level of its first-order sensitivity with respect to the population distribution. A central ingredient of our analysis is the derivation of a nonlocal parabolic evolution equation satisfied by the functional measure derivative along a characteristic measure flow. This formulation makes it possible to apply Carleman techniques directly at the level of the measure derivative. By combining the resulting estimate with the stability estimate for the master field and the associated measure dynamics, we establish a terminal-to-interior Lipschitz stability result for both the master field and its functional measure derivative. In this way, the present work develops a Carleman-based framework for the stability analysis of the master equation, and further extends this framework to its measure-variable regularity.

\subsection{Problem statement and main results.}
Next, we formulate the master equation under consideration and state the
stability problem studied in this paper. Let $\mathbb{T}^d$ denote the
$d$-dimensional flat torus and let $\mathcal{P}(\mathbb{T}^d)$ be the space
of Borel probability measures on $\mathbb{T}^d$. We consider the following
mean field game master equation without common noise
\begin{equation}\label{MasterEquation}
\left\{
\begin{aligned}
&-\partial_t U(t,x,m)
-\Delta_x U(t,x,m)
+H\bigl(x,D_xU(t,x,m)\bigr)
\\
&\qquad
-\int_{\mathbb T^d}
\mathrm{div}_y\bigl(D_mU(t,x,m,y)\bigr)\,dm(y)
\\
&\qquad
+\int_{\mathbb T^d}
D_mU(t,x,m,y)\cdot
D_pH\bigl(y,D_xU(t,y,m)\bigr)\mathrm{d}m(y)
=F(x,m),\\
&\qquad
(x,t,m)\in \mathbb T^d\times(0,T)\times\mathcal P(\mathbb T^d),
\\
&U(T,x,m)=G(x,m), \qquad
(x,m)\in \mathbb T^d\times\mathcal P(\mathbb T^d),
\end{aligned}
\right.
\end{equation}
Here $U(t,x,m)$ denotes the value of a representative agent at time $t$
and state $x$ when the population distribution is $m$. The function
$H(x,p)$ is the Hamiltonian, while $F(x,m)$ and $G(x,m)$ denote the
running and terminal cost functionals, respectively. To specify the measure derivatives appearing in \eqref{MasterEquation}, we recall
the following notation. Let
$
\Phi:\mathcal{P}(\mathbb{T}^d)\to\mathbb{R}$, and we say that $\Phi$ is differentiable with respect to the measure variable if there exists a continuous function
\[
\frac{\delta\Phi}{\delta m}:
\mathcal{P}(\mathbb{T}^d)\times\mathbb{T}^d\to\mathbb{R}
\]
such that, for any $m,m'\in\mathcal{P}(\mathbb{T}^d)$,
\[
\lim_{s\to0^+}
\frac{\Phi((1-s)m+sm')-\Phi(m)}{s}
=
\int_{\mathbb{T}^d}
\frac{\delta\Phi}{\delta m}(m,y)\,d(m'-m)(y).
\]
The functional derivative is uniquely determined up to an additive constant. Throughout
this paper, we adopt the normalization
\[
\int_{\mathbb{T}^d}
\frac{\delta\Phi}{\delta m}(m,y)\,dm(y)=0.
\]
If the map
$
y\longmapsto\frac{\delta\Phi}{\delta m}(m,y)$
is differentiable, we define the intrinsic derivative, also called the Lions derivative by
\[
D_m\Phi(m,y)
:=
D_y\frac{\delta\Phi}{\delta m}(m,y).
\]
Applying these definitions to $\Phi(m)=U(t,x,m)$ gives
$\delta U/\delta m$ and $D_mU$ appearing throughout the paper.

Since the measure variable is infinite-dimensional, it is useful to
analyze the master equation along measure flows generated by the
associated mean field game system. For a given initial time
$t_0\in[0,T]$ and initial distribution
$m_0\in\mathcal{P}(\mathbb{T}^d)$, consider the following mean field game system
\begin{equation}\label{MFG}
\left\{
\begin{aligned}
&-\partial_t u(t,x)-\Delta u(t,x)
+H\bigl(x,D_xu(t,x)\bigr)
=F(x,m),
\quad (t,x)\in(t_0,T)\times\mathbb T^d,
\\
&\partial_t m(t,x)-\Delta m(t,x)
-\operatorname{div}\Bigl(
m(t,x)D_pH\bigl(x,D_xu(t,x)\bigr)
\Bigr)=0,
\quad (t,x)\in(t_0,T)\times\mathbb T^d,
\\
&u(T,x)=G(x,m(T)), \quad
m(t_0,x)=m_0(x),
\quad x\in\mathbb T^d.
\end{aligned}
\right.
\end{equation}
When the measure flow is absolutely continuous, we use the same notation
$m(t,x)$ for its density. For $M>0$, we define the class of admissible probability densities
\begin{equation*}\label{eq:bounded-density-class}
\mathcal P_M(\mathbb T^d)
:=
\left\{
m\in \mathcal P(\mathbb T^d)\cap L^\infty(\mathbb T^d):
\|m\|_{L^\infty(\mathbb T^d)}\le M
\right\}.
\end{equation*}
If $U$ is a sufficiently regular solution of \eqref{MasterEquation}
and $(u,m)$ is the corresponding solution of \eqref{MFG}, then
\begin{equation}
\label{eq:MFG-slice}
u(t,x)=U(t,x,m(t)),
\qquad t\in[t_0,T].
\end{equation}
Following the characteristic interpretation of the MFG system for the master equation,
we refer to the measure component $m$ of \eqref{MFG} as a
characteristic measure flow associated with $U$.
Thus, \eqref{eq:MFG-slice} describes the restriction of the master field along such a
characteristic measure flow. This relation, which is recalled rigorously in Section~2,
provides the link between the master equation and the finite-dimensional parabolic
systems used in our stability analysis.

We now formulate the stability problem considered in this paper.
Let $U_1$ and $U_2$ be two sufficiently regular classical solutions
of the master equation corresponding to the same Hamiltonian $H$
and running cost $F$, but to possibly different terminal cost
functionals $G_1$ and $G_2$, respectively. We set
$\overline{U}:=U_1-U_2$.
Let $M>0$ be fixed and let
$m_0\in\mathcal P_M(\mathbb T^d)$ be an arbitrary initial probability
density. Let $m_1(t)$ denote the characteristic measure flow associated
with $U_1$ and issued from $m_0$. The initial density $m_0$ is not fixed
throughout the analysis; rather, the argument applies to any
$m_0\in\mathcal P_M(\mathbb T^d)$ for which the corresponding
characteristic flows satisfy the assumptions imposed below.
Accordingly, $m_1$ should be regarded as an arbitrary reference
characteristic flow within this class.

The purpose of the present work is to quantify the propagation of
perturbations in the terminal data to the master field and, in
particular, to its functional derivative with respect to the measure
variable. More precisely, under the assumptions stated in the
subsequent sections, we establish a terminal-to-interior estimate of
the form
\begin{equation}
\label{eq:intro-main-stability}
\begin{aligned}
&
\|\overline{U}(\cdot,\cdot,m_1(\cdot))\|
_{L^2(0,T;H^1(\mathbb{T}^d))}^2
+
\left\|
\frac{\delta\overline{U}}{\delta m}
(\cdot,\cdot,m_1(\cdot),\cdot)
\right\|_
{L^2(0,T;H^1(\mathbb{T}^d\times\mathbb{T}^d))}^2
\\
&\leq C
\left[
\|\overline{U}(T,\cdot,m_1(T))\|_{H^1(\mathbb{T}^d)}^2
+
\left\|
\frac{\delta\overline{U}}{\delta m}
(T,\cdot,m_1(T),\cdot)
\right\|_
{H^1(\mathbb{T}^d\times\mathbb{T}^d)}^2
\right].
\end{aligned}
\end{equation}
Since
\[
U_i(T,x,m)=G_i(x,m),
\qquad
\frac{\delta U_i}{\delta m}(T,x,m,y)
=
\frac{\delta G_i}{\delta m}(x,m,y),
\qquad i=1,2,
\]
for every $m\in\mathcal P_M(\mathbb T^d)$, the quantities appearing on
the right-hand side of \eqref{eq:intro-main-stability} are determined by
the prescribed terminal functionals $G_1$ and $G_2$ and their functional
derivatives, evaluated at the terminal point $m_1(T)$ of the chosen
characteristic flow. Consequently, the estimate applies to any reference
characteristic flow issued from an initial density
$m_0\in\mathcal P_M(\mathbb T^d)$ and satisfying the assumptions imposed
below. In this sense, the result is not tied to a
particular population trajectory, although the stability estimate is
formulated along a characteristic flow in the measure space. The
principal object of interest is the functional measure derivative,
while the estimate for $\overline{U}$ provides the auxiliary control
required to close the stability analysis.

The outline of this paper is as follows.
In Section \ref{section2}, we recall the connection between the master equation and the
associated mean field game system, and characterize the functional measure
derivative through the corresponding linearized MFG system. Based on this
representation, we derive a nonlocal parabolic evolution equation satisfied by
the functional measure derivative along a characteristic measure flow.
In Section \ref{section3}, we establish the stability estimates for the difference of
two master fields. We introduce the derivative along a measure
flow and derive a Carleman estimate for the master-field difference restricted
to a reference characteristic measure flow.
In Section \ref{section4}, we derive the
equation satisfied by the difference of two measure derivatives along their
respective characteristic flows, establish the corresponding Carleman estimate,
and control the separation of the associated measure flows. Combining these
estimates, we obtain the terminal-to-interior Lipschitz stability estimate for
the functional measure derivative, which constitutes the main result of the
paper.

\section{The equation for the functional measure derivative}\label{section2}

We first recall the representation of the solution to the master equation
through the associated mean field game system. This representation is the
starting point for deriving the equation satisfied by the measure derivative
of the master equation.

\begin{lem}\label{lem:ME-MFG}
	Let $U=U(t,x,m)$ be a classical solution (see \cite[Definition~2.7]{CardaliaguetDelarueLasryLions2019}) of the master equation
	\eqref{MasterEquation}. Fix	
	$(t_0,m_0)\in[0,T]\times\mathcal{P}(\mathbb{T}^d)$.
	Let $m=m(t)$ be a characteristic measure flow associated with $U$ in \eqref{MFG}
	with $m(t_0)=m_0$, and define
	\begin{equation}\label{u restriction}
	u(t,x):=U(t,x,m(t)).
	\end{equation}
	Then $u$ satisfies the Hamilton--Jacobi equation in \eqref{MFG}. In particular,
	\begin{equation*}\label{U-u relation}
	U(t_0,x,m_0)=u(t_0,x).
	\end{equation*}
\end{lem}

\begin{proof}
	By the definition \eqref{u restriction}, we have
	\[
	D_xu(t,x)=D_xU(t,x,m(t)),
	\qquad
	\Delta u(t,x)=\Delta_xU(t,x,m(t)).
	\]
	Moreover, by the chain rule with respect to the measure variable, we have
	\begin{equation}\label{chain rule U}
	\begin{aligned}
	\partial_tu(t,x)
	={}
	\partial_tU(t,x,m(t))
	+
	\int_{\mathbb{T}^d}
	\frac{\delta U}{\delta m}
	(t,x,m(t),y)
	\,\partial_tm(t,y)\,dy.
	\end{aligned}
	\end{equation}
	By \eqref{MFG},
	\[
	\partial_tm(t,y)
	=
	\Delta_y m(t,y)
	+
	\operatorname{div}_y\left(
	m(t,y)
	D_pH\left(y,D_yU(t,y,m(t))\right)
	\right).
	\]
	Substituting this identity into \eqref{chain rule U} and integrating by
	parts on $\mathbb{T}^d$, we obtain
	\begin{align}
	&
	\int_{\mathbb{T}^d}
	\frac{\delta U}{\delta m}
	(t,x,m(t),y)
	\,\partial_tm(t,y)\,dy
	\nonumber\\
	={}&
	\int_{\mathbb{T}^d}
	\frac{\delta U}{\delta m}
	(t,x,m(t),y)
	\,\Delta_y m(t,y)\,dy
	\nonumber\\
	&+
	\int_{\mathbb{T}^d}
	\frac{\delta U}{\delta m}
	(t,x,m(t),y)
	\operatorname{div}_y\left(
	m(t,y)
	D_pH\left(y,D_yU(t,y,m(t))\right)
	\right)\,dy
	\nonumber\\
	={}&
	\int_{\mathbb{T}^d}
	\operatorname{div}_y
	\left(
	D_mU(t,x,m(t),y)
	\right)
	\,dm(t,y)
	\nonumber\\
	&-
	\int_{\mathbb{T}^d}
	D_mU(t,x,m(t),y)
	\cdot
	D_pH\left(
	y,D_yU(t,y,m(t))
	\right)
	\,dm(t,y).
	\label{measure chain term}
	\end{align}

    Substituting \eqref{measure chain term} into
    \eqref{chain rule U}, and using the master equation evaluated at
    $(t,x,m(t))$, gives
    \[
    -\partial_t u-\Delta u+H(x,D_xu)=F(x,m(t)).
    \]
    Finally, the terminal condition follows from the terminal condition of the
    master equation
    \[
    u(T,x)=U(T,x,m(T))=G(x,m(T)).
    \]
    This proves the lemma.
\end{proof}

We next recall the relationship between the functional measure derivative
of the master field and the linearized MFG system, see \cite{CardaliaguetDelarueLasryLions2019}.
\begin{lem}
	\label{lem:linearized-characterization-deltaU}
	Let $U=U(t,x,m)$ be the classical solution of the master equation, and fix
	$(t_0,m_0)\in[0,T]\times\mathcal P(\mathbb T^d)$. Let $(u,m)$ be the solution to the corresponding MFG system with initial condition $m(t_0)=m_0$.
	
	For an arbitrary perturbation measure $\mu_0$ with zero total mass, let $(v,\mu)$ solve the linearized MFG
	system
	\begin{equation}\label{linearized-MFG}
	\left\{
	\begin{aligned}
	&-\partial_t v-\Delta v
	+D_pH(x,D_xu)\cdot D_xv=
	\frac{\delta F}{\delta m}(x,m(t))(\mu(t)),
	\\
	&\partial_t\mu-\Delta\mu
	-\operatorname{div}\bigl(\mu D_pH(x,D_xu)\bigr)
	-\operatorname{div}\bigl(mD^2_{pp}H(x,D_xu)D_xv\bigr)=0,
	\\
	&v(T,x)=
	\frac{\delta G}{\delta m}(x,m(T))(\mu(T)),
	x\in\mathbb T^d,
	\\
	&\mu(t_0,x)=\mu_0(x),
	x\in\mathbb T^d.
	\end{aligned}
	\right.
	\end{equation}
	Here
	\begin{equation*}
	\frac{\delta F}{\delta m}(x,m(t))(\mu(t))
	:=
	\int_{\mathbb T^d}
	\frac{\delta F}{\delta m}(x,m(t),y)\,\mu(t,y)\,dy,
	\end{equation*}
	and the term involving $\delta G/\delta m$ is understood in the same sense.
	
	Then the measure derivative of the master equation solution is characterized by
	\begin{equation*}
	v(t_0,x)
	=
	\int_{\mathbb T^d}
	\frac{\delta U}{\delta m}(t_0,x,m_0,y)\,\mu_0(y)\,dy.
	\end{equation*}
\end{lem}

Lemma \ref{lem:linearized-characterization-deltaU} establishes a fundamental connection between the master equation and the associated linearized MFG system. It shows that the derivative of the master equation solution with respect to the measure variable can be characterized by the response of the MFG system to perturbations of the initial distribution.

Since every point $(t,m(t))$ along the flow of measures generated by the MFG system may be viewed as a new initial state of the master equation, the representation formula of Lemma
\ref{lem:linearized-characterization-deltaU}
extends immediately from $(t_0,m_0)$ to $(t,m(t))$.

\begin{cor}\label{deltaUm}
	\label{cor:deltaU-along-flow}
	Let $(u,m)$ be the solution to the MFG system generated by the master equation \eqref{MasterEquation}.
	Then for any $t\in[0,T]$, one has
	\begin{equation*}
	v(t,x)=
	\int_{\mathbb T^d}
	\frac{\delta U}{\delta m}
	\bigl(t,x,m(t),y\bigr)
	\mu(t,y)\,dy.
	\label{eq:deltaU-along-trajectory}
	\end{equation*}
	where $(v,\mu)$ denotes the solution of the linearized MFG system \eqref{linearized-MFG}.
\end{cor}

\begin{rem}
	This observation is of particular importance for our subsequent analysis. It allows us to represent the measure derivative by an integral kernel and to study its properties through the corresponding linearized partial differential equations.
\end{rem}

Motivated by Corollary \ref{deltaUm}, we introduce
\begin{equation}
K(t,x,y):=
\frac{\delta U}{\delta m}(t,x,m(t),y),
\qquad (t,x,y)\in[0,T]\times\mathbb T^d\times\mathbb T^d .
\label{eq:K-definition}
\end{equation}
Here $m(t)$ is the population distribution along the flow of measures generated by the MFG system \eqref{MFG}.
Thus $K$ is the restriction of the measure derivative of the master equation
to the curve $t\mapsto m(t)$.

For simplicity, we denote
\begin{equation*}
b(t,x):=D_pH\bigl(x,D_xu(t,x)\bigr),
\qquad
A(t,x):=D^2_{pp}H\bigl(x,D_xu(t,x)\bigr).
\label{eq:a-B-definition}
\end{equation*}

\begin{thm}
	\label{prop:equation-for-K}
	Assume that the master equation \eqref{MasterEquation} admits a classical solution $U(t,x,m)$ and that the corresponding MFG system \eqref{MFG} admits the solution $(u,m)$ with its linearized system \eqref{linearized-MFG}.
	Then the function $K$ defined by \eqref{eq:K-definition} satisfies
	\begin{equation}
	\left\{
	\begin{aligned}
	&-\partial_t K(t,x,y)
	-\Delta_x K(t,x,y)
	-\Delta_y K(t,x,y)
	+b(t,x)\cdot D_xK(t,x,y)
	+b(t,y)\cdot D_yK(t,x,y)
	\\
	&\quad
	+\int_{\mathbb T^d}
	m(t,z)\,
	D_yK(t,x,z)\cdot
	A(t,z)D_xK(t,z,y)\,dz=
	\frac{\delta F}{\delta m}(x,m(t),y),	
	\\
	&K(T,x,y)=
	\frac{\delta G}{\delta m}(x,m(T),y),
	\qquad (x,y)\in\mathbb T^d\times\mathbb T^d.
	\end{aligned}
	\right.
	\label{eq:K-equation}
	\end{equation}
	for $(t,x,y)\in(0,T)\times\mathbb T^d\times\mathbb T^d$,
	with the normalization of the functional derivative
	\begin{equation}
	\int_{\mathbb T^d}K(t,x,y)\,dm(t,y)=0.
	\label{eq:K-normalization}
	\end{equation}
\end{thm}

\begin{proof}
	Let $(v,\mu)$ be the solution of the linearized MFG system associated with an
	admissible perturbation of the initial distribution. By Corollary \ref{deltaUm}, for every
	$t\in[0,T]$,
	\begin{equation}\label{eq:v-K-mu}
	v(t,x)
	=
	\int_{\mathbb T^d}
	K(t,x,y)\mu(t,y)\,dy.
	\end{equation}
	Differentiating \eqref{eq:v-K-mu} with respect to $t$, we obtain
	\begin{equation*}\label{eq:vt-K}
	\partial_t v(t,x)
	=
	\int_{\mathbb T^d}
	\partial_tK(t,x,y)\mu(t,y)\,dy
	+
	\int_{\mathbb T^d}
	K(t,x,y)\partial_t\mu(t,y)\,dy.
	\end{equation*}
	
	By the linearized Fokker--Planck equation in \eqref{linearized-MFG},
	\begin{equation*}\label{eq:mu-evolution}
	\partial_t\mu
	=
	\Delta\mu
	+
	\operatorname{div}(\mu b)
	+
	\operatorname{div}\bigl(mA D_xv\bigr).
	\end{equation*}
	Using integration by parts on $\mathbb T^d$, we obtain
	\begin{equation}\label{eq:K-mu-t}
	\begin{aligned}
	-
	\int_{\mathbb T^d}
	K(t,x,y)\partial_t\mu(t,y)\,dy
	={}&
	-\int_{\mathbb T^d}
	\Delta_yK(t,x,y)\mu(t,y)\,dy
	\\
	&+
	\int_{\mathbb T^d}
	D_yK(t,x,y)\cdot b(t,y)\mu(t,y)\,dy
	\\
	&+
	\int_{\mathbb T^d}
	D_yK(t,x,y)\cdot
	m(t,y)A(t,y)D_xv(t,y)\,dy.
	\end{aligned}
	\end{equation}
	
	Moreover, from \eqref{eq:v-K-mu},
	\begin{equation}\label{eq:Dv-K}
	D_xv(t,y)
	=
	\int_{\mathbb T^d}
	D_xK(t,y,z)\mu(t,z)\,dz.
	\end{equation}
	Substituting \eqref{eq:Dv-K} into the last term of
	\eqref{eq:K-mu-t} gives
	\begin{equation}\label{eq:nonlocal-K-term}
	\begin{aligned}
	&
	\int_{\mathbb T^d}
	D_yK(t,x,y)\cdot
	m(t,y)A(t,y)D_xv(t,y)\,dy
	\\
	&\qquad
	=
	\int_{\mathbb T^d}
	\left[
	\int_{\mathbb T^d}
	m(t,z)
	D_yK(t,x,z)\cdot
	A(t,z)D_xK(t,z,y)\,dz
	\right]
	\mu(t,y)\,dy.
	\end{aligned}
	\end{equation}
	
	On the other hand, the linearized Hamilton--Jacobi equation is
	\begin{equation}\label{eq:linearized-HJ-short}
	-\partial_tv(t,x)
	-\Delta_xv(t,x)
	+b(t,x)\cdot D_xv(t,x)
	=
	\int_{\mathbb T^d}
	\frac{\delta F}{\delta m}
	\bigl(x,m(t),y\bigr)\mu(t,y)\,dy.
	\end{equation}
	Using \eqref{eq:v-K-mu}, \eqref{eq:K-mu-t}, and
	\eqref{eq:nonlocal-K-term}, equation
	\eqref{eq:linearized-HJ-short} can be rewritten as
	\begin{equation}\label{eq:E-mu-identity}
	\int_{\mathbb T^d}
	\mathcal E(t,x,y)\mu(t,y)\,dy=0,
	\end{equation}
	where
	\begin{equation}\label{eq:E-definition}
	\begin{aligned}
	\mathcal E(t,x,y)
	:={}&
	-\partial_tK(t,x,y)
	-\Delta_xK(t,x,y)
	-\Delta_yK(t,x,y)
	\\
	&+
	b(t,x)\cdot D_xK(t,x,y)
	+
	b(t,y)\cdot D_yK(t,x,y)
	\\
	&+
	\int_{\mathbb T^d}
	m(t,z)
	D_yK(t,x,z)\cdot
	A(t,z)D_xK(t,z,y)\,dz
	\\
	&-
	\frac{\delta F}{\delta m}
	\bigl(x,m(t),y\bigr).
	\end{aligned}
	\end{equation}
	
	Fix an arbitrary $\tau\in(0,T)$. Since the point
	$(\tau,m(\tau))$ may be regarded as a new initial state of the MFG system,
	we may consider the linearized system on $[\tau,T]$ with an arbitrary
	admissible perturbation $\eta$ satisfying
	\begin{equation}\label{eq:eta-zero-mass}
	\int_{\mathbb T^d}\eta(y)\,dy=0
	\end{equation}
	and
	\begin{equation*}\label{eq:mu-tau-eta}
	\mu(\tau,\cdot)=\eta.
	\end{equation*}
	Evaluating \eqref{eq:E-mu-identity} at $t=\tau$ therefore yields
	\begin{equation*}\label{eq:E-zero-mass}
	\int_{\mathbb T^d}
	\mathcal E(\tau,x,y)\eta(y)\,dy=0
	\end{equation*}
	for every admissible perturbation $\eta$ satisfying
	\eqref{eq:eta-zero-mass}. Hence
	$\mathcal E(\tau,x,\cdot)$ is independent of the variable $y$.
	Since $\tau\in(0,T)$ is arbitrary, there exists a function $c=c(t,x)$
	such that
	\begin{equation}\label{eq:E-constant}
	\mathcal E(t,x,y)=c(t,x),
	\qquad
	(t,x,y)\in(0,T)\times\mathbb T^d\times\mathbb T^d.
	\end{equation}
	
    We can also prove that $c(t,x)=0$. By the normalization
	\eqref{eq:K-normalization},
	\begin{equation}\label{eq:normalization-time}
	\int_{\mathbb T^d}
	K(t,x,y)m(t,y)\,dy=0.
	\end{equation}
	Differentiating \eqref{eq:normalization-time} with respect to $t$ and using
	the Fokker--Planck equation
	\begin{equation*}\label{eq:m-evolution-for-normalization}
	\partial_tm
	=
	\Delta m+\operatorname{div}(mb),
	\end{equation*}
	we obtain
	\begin{equation*}\label{eq:normalization-time-derivative}
	\begin{aligned}
	0
	={}&
	\int_{\mathbb T^d}
	\partial_tK(t,x,y)m(t,y)\,dy
	+
	\int_{\mathbb T^d}
	K(t,x,y)\partial_tm(t,y)\,dy
	\\
	={}&
	\int_{\mathbb T^d}
	\Bigl(
	\partial_tK(t,x,y)
	+\Delta_yK(t,x,y)
	-b(t,y)\cdot D_yK(t,x,y)
	\Bigr)
	m(t,y)\,dy.
	\end{aligned}
	\end{equation*}
	Therefore,
	\begin{equation}\label{eq:y-part-zero}
	\int_{\mathbb T^d}
	\Bigl(
	-\partial_tK
	-\Delta_yK
	+b(t,y)\cdot D_yK
	\Bigr)(t,x,y)\,dm(t,y)
	=0.
	\end{equation}
	
	On the other hand, differentiating \eqref{eq:K-normalization} with respect
	to $x$ gives
	\begin{equation}\label{eq:x-normalization}
	\int_{\mathbb T^d}
	D_xK(t,x,y)\,dm(t,y)=0,
	\qquad
	\int_{\mathbb T^d}
	\Delta_xK(t,x,y)\,dm(t,y)=0.
	\end{equation}
	Consequently,
	\begin{equation}\label{eq:x-part-zero}
	\int_{\mathbb T^d}
	\Bigl(
	-\Delta_xK(t,x,y)
	+b(t,x)\cdot D_xK(t,x,y)
	\Bigr)\,dm(t,y)
	=0.
	\end{equation}
	
	For the nonlocal term, by Fubini's theorem and
	\eqref{eq:x-normalization},
	\begin{equation}\label{eq:nonlocal-normalization-zero}
	\begin{aligned}
	&
	\int_{\mathbb T^d}
	\left[
	\int_{\mathbb T^d}
	m(t,z)
	D_yK(t,x,z)\cdot
	A(t,z)D_xK(t,z,y)\,dz
	\right]dm(t,y)
	\\
	&\qquad
	=
	\int_{\mathbb T^d}
	m(t,z)D_yK(t,x,z)\cdot A(t,z)
	\left[
	\int_{\mathbb T^d}
	D_xK(t,z,y)\,dm(t,y)
	\right]dz
	=0.
	\end{aligned}
	\end{equation}
	Furthermore, by the normalization of the measure derivative, we have
	\begin{equation}\label{eq:F-normalization-along-flow}
	\int_{\mathbb T^d}
	\frac{\delta F}{\delta m}
	\bigl(x,m(t),y\bigr)\,dm(t,y)=0.
	\end{equation}
	
	Integrating \eqref{eq:E-constant} with respect to $m(t)$ and using
	\eqref{eq:y-part-zero}, \eqref{eq:x-part-zero},
	\eqref{eq:nonlocal-normalization-zero}, and
	\eqref{eq:F-normalization-along-flow}, we obtain
	\begin{equation*}\label{eq:c-zero}
	c(t,x)=0.
	\end{equation*}
	Hence $\mathcal E(t,x,y)=0$, and by
	\eqref{eq:E-definition} we obtain precisely
	\eqref{eq:K-equation}.
	
	Finally, since
	\begin{equation*}\label{eq:terminal-U-G}
	U(T,x,m)=G(x,m)
	\end{equation*}
	for every admissible $m$, their functional derivatives with respect to
	the measure variable coincide up to an additive constant in $y$.
	Since both functional derivatives are taken with the same normalization
	convention, this constant is zero. Therefore, we obtain that
	\begin{equation*}\label{eq:terminal-K-identification}
	K(T,x,y)
	=
	\frac{\delta U}{\delta m}
	\bigl(T,x,m(T),y\bigr)
	=
	\frac{\delta G}{\delta m}
	\bigl(x,m(T),y\bigr).
	\end{equation*}
    The normalization
	\eqref{eq:K-normalization} follows directly from the definition
	\eqref{eq:K-definition} and the normalization of
	$\delta U/\delta m$. This completes the proof.
\end{proof}

\section{Lipschitz stability estimates for the master field}\label{section3}

We next introduce the derivative of a functional along a measure flow.
This identity will be useful in the subsequent analysis of the master
equation along the associated flow of probability measures.

\begin{lem}\label{lem:characteristic-derivative}
	Let
	$V:[t_0,T]\times\mathbb{T}^d\times\mathcal{P}(\mathbb{T}^d)
	\longrightarrow\mathbb{R}$
	be continuously differentiable in $t$ and of class $C^1$ with respect
	to the measure variable $m$. Assume that its functional derivative
	$\delta V/\delta m$ is twice differentiable with respect to $y$, the
	corresponding derivatives being continuous in all variables.
	Let
	$\beta=\beta(t,y,m)$
	be a vector field continuous in all variables and continuously
	differentiable with respect to $y$. Let $m$ be a flow of
	probability measures, satisfying
	\begin{equation}\label{eq:general-measure-flow}
	\begin{cases}
	\displaystyle
	\partial_t m(t,y)-\Delta_y m(t,y)
	-\operatorname{div}_y
	\left(
	m\cdot\beta(t,y,m(t))
	\right)=0,
	&
	(t,y)\in(t_0,T)\times\mathbb{T}^d,
	\\[1mm]
	m(t_0,y)=m_0(y),
	&
	y\in\mathbb{T}^d.
	\end{cases}
	\end{equation}
	
	Define
	\begin{equation*}\label{eq:characteristic-derivative}
	\begin{aligned}
	\mathcal{D}_{\beta}V(t,x,m)
	:={}&
	\partial_tV(t,x,m)
	+
	\int_{\mathbb{T}^d}
	\operatorname{div}_y
	\left(
	D_mV(t,x,m,y)
	\right)
	\,dm(y)
	\\
	&-
	\int_{\mathbb{T}^d}
	D_mV(t,x,m,y)
	\cdot
	\beta(t,y,m)
	\,dm(y).
	\end{aligned}
	\end{equation*}
	Then, along the measure flow $m(t)$, we have that
	\begin{equation*}\label{eq:characteristic-derivative-identity}
	\frac{d}{dt}V(t,x,m(t))
	=
	\mathcal{D}_{\beta}V(t,x,m(t)),
	\qquad
	(t,x)\in(t_0,T)\times\mathbb{T}^d.
	\end{equation*}
\end{lem}

\begin{proof}
	The proof follows from the same chain-rule argument used in
	Lemma~\ref{lem:ME-MFG}, now with a general drift field $\beta$.
	Indeed,
	\[
	\frac{d}{dt}V(t,x,m(t))
	=
	\partial_tV(t,x,m(t))
	+
	\int_{\mathbb{T}^d}
	\frac{\delta V}{\delta m}(t,x,m(t),y)
	\,\partial_tm(t,y)\,dy.
	\]
	Using \eqref{eq:general-measure-flow} and integrating by parts on
	$\mathbb{T}^d$, we obtain
	\[
	\begin{aligned}
	\int_{\mathbb{T}^d}
	\frac{\delta V}{\delta m}
	\,\partial_tm\,dy
	={}&
	\int_{\mathbb{T}^d}
	\operatorname{div}_y
	\left(D_mV(t,x,m(t),y)\right)
	\,dm(t,y)
	\\
	&-
	\int_{\mathbb{T}^d}
	D_mV(t,x,m(t),y)
	\cdot
	\beta(t,y,m(t))
	\,dm(t,y).
	\end{aligned}
	\]
	Substituting this identity into the chain rule and using the definition
	of $\mathcal D_\beta$ yields
	\[
	\frac{d}{dt}V(t,x,m(t))
	=
	\mathcal D_\beta V(t,x,m(t)).
	\]

\end{proof}

We now compare two solutions of the master equation directly at the same
measure argument. This allows us to retain the full structure of the
master equation before restricting it to an associated characteristic
measure flow.

\begin{prop}\label{prop:master-difference-equation}
	Let $U_1$ and $U_2$ be two classical solutions of the master equation
	\eqref{MasterEquation}, corresponding to the same Hamiltonian $H$ and
	running cost $F$, with terminal conditions
	\[
	U_i(T,x,m)=G_i(x,m),
	\qquad i=1,2.
	\]
	Set
	\begin{equation*}\label{eq:master-difference-U}
	\overline U(t,x,m)
	:=
	U_1(t,x,m)-U_2(t,x,m).
	\end{equation*}
	For $i=1,2$, define
	\begin{equation*}\label{eq:master-beta-i}
	\beta_i(t,y,m)
	:=
	D_pH\left(y,D_yU_i(t,y,m)\right).
	\end{equation*}
	Define
	\begin{equation*}\label{eq:master-b-hat}
	\widehat b(t,x,m)
	:=
	\int_0^1
	D_pH\left(
	x,
	D_xU_2(t,x,m)
	+sD_x\overline U(t,x,m)
	\right)\,ds,
	\end{equation*}
	\begin{equation*}\label{eq:master-A-hat}
	\widehat A(t,y,m)
	:=
	\int_0^1
	D_{pp}^2H\left(
	y,
	D_yU_2(t,y,m)
	+sD_y\overline U(t,y,m)
	\right)\,ds.
	\end{equation*}
	Then $\overline U$ satisfies
	\begin{align}
	-\mathcal D_{\beta_1}\overline U(t,x,m)
	&-\Delta_x\overline U(t,x,m)
	+
	\widehat b(t,x,m)\cdot
	D_x\overline U(t,x,m)
	\nonumber\\
	&+
	\int_{\mathbb T^d}
	D_mU_2(t,x,m,y)
	\cdot
	\widehat A(t,y,m)
	D_y\overline U(t,y,m)
	\,dm(y)
	=0
	\label{eq:full-master-difference}
	\end{align}
	for
	$(t,x,m)\in
	(0,T)\times\mathbb T^d\times\mathcal P(\mathbb T^d)$,
	with terminal condition
	\begin{equation}\label{eq:full-master-difference-terminal}
	\overline U(T,x,m)
	=
	G_1(x,m)-G_2(x,m).
	\end{equation}
	Here $\mathcal D_{\beta_1}$ is the derivative operator introduced in
	Lemma~\ref{lem:characteristic-derivative}.
\end{prop}

\begin{proof}
	Subtracting the equations satisfied by $U_1$ and $U_2$, we obtain
	\begin{align}
	0
	={}&
	-\partial_t\overline U
	-\Delta_x\overline U
	+
	H\left(x,D_xU_1\right)
	-
	H\left(x,D_xU_2\right)
	\nonumber\\
	&-
	\int_{\mathbb T^d}
	\operatorname{div}_y
	\left(
	D_m\overline U(t,x,m,y)
	\right)
	\,dm(y)
	\nonumber\\
	&+
	\int_{\mathbb T^d}
	\Big[
	D_mU_1(t,x,m,y)\cdot\beta_1(t,y,m)
	-
	D_mU_2(t,x,m,y)\cdot\beta_2(t,y,m)
	\Big]
	\,dm(y).
	\label{eq:master-difference-first}
	\end{align}
	
	We first consider the Hamiltonian term. By the mean value theorem, we have
	\begin{align}
	&
	H\left(x,D_xU_1(t,x,m)\right)
	-
	H\left(x,D_xU_2(t,x,m)\right)
	\nonumber\\
	={}&
	\int_0^1
	D_pH\left(
	x,
	D_xU_2(t,x,m)
	+sD_x\overline U(t,x,m)
	\right)
	\cdot
	D_x\overline U(t,x,m)
	\,ds
	\nonumber\\
	={}&
	\widehat b(t,x,m)
	\cdot
	D_x\overline U(t,x,m).
	\label{eq:H-difference}
	\end{align}
	
	Similarly, applying the mean value theorem to
	$p\mapsto D_pH(y,p)$, we obtain
    \begin{equation*}
	\begin{aligned}
	\beta_1(t,y,m)-\beta_2(t,y,m)
	={}&
	\int_0^1
	D_{pp}^2H\left(
	y,
	D_yU_2(t,y,m)
	+sD_y\overline U(t,y,m)
	\right)
	D_y\overline U(t,y,m)
	\,ds
	\\
	={}&
	\widehat A(t,y,m)
	D_y\overline U(t,y,m).
	\label{eq:beta-difference-master}
	\end{aligned}
	\end{equation*}
	
	For the last integral term in
	\eqref{eq:master-difference-first}, we write
	\begin{align}
	&
	D_mU_1(t,x,m,y)\cdot\beta_1(t,y,m)
	-
	D_mU_2(t,x,m,y)\cdot\beta_2(t,y,m)
	\nonumber\\
	={}&
	\left(
	D_mU_1(t,x,m,y)
	-
	D_mU_2(t,x,m,y)
	\right)
	\cdot\beta_1(t,y,m)
	\nonumber\\
	&+
	D_mU_2(t,x,m,y)
	\cdot
	\left(
	\beta_1(t,y,m)-\beta_2(t,y,m)
	\right)
	\nonumber\\
	={}&
	D_m\overline U(t,x,m,y)
	\cdot
	\beta_1(t,y,m)
	+
	D_mU_2(t,x,m,y)
	\cdot
	\widehat A(t,y,m)
	D_y\overline U(t,y,m).
	\label{eq:measure-term-difference}
	\end{align}
	
	Substituting
	\eqref{eq:H-difference} and
	\eqref{eq:measure-term-difference}
	into \eqref{eq:master-difference-first}, we obtain
	\begin{align}
	0
	={}&
	-\partial_t\overline U
	-\Delta_x\overline U
	+
	\widehat b(t,x,m)
	\cdot D_x\overline U
	\nonumber\\
	&-
	\int_{\mathbb T^d}
	\operatorname{div}_y
	\left(
	D_m\overline U(t,x,m,y)
	\right)
	\,dm(y)
	\nonumber\\
	&+
	\int_{\mathbb T^d}
	D_m\overline U(t,x,m,y)
	\cdot
	\beta_1(t,y,m)
	\,dm(y)
	\nonumber\\
	&+
	\int_{\mathbb T^d}
	D_mU_2(t,x,m,y)
	\cdot
	\widehat A(t,y,m)
	D_y\overline U(t,y,m)
	\,dm(y).
	\label{eq:master-difference-expanded}
	\end{align}
	
	By the definition of $\mathcal D_{\beta_1}$ in
	Lemma~\ref{lem:characteristic-derivative},
	\begin{align*}
	-\mathcal D_{\beta_1}\overline U(t,x,m)
	={}&
	-\partial_t\overline U(t,x,m)
	\\
	&-
	\int_{\mathbb T^d}
	\operatorname{div}_y
	\left(
	D_m\overline U(t,x,m,y)
	\right)
	\,dm(y)
	\\
	&+
	\int_{\mathbb T^d}
	D_m\overline U(t,x,m,y)
	\cdot
	\beta_1(t,y,m)
	\,dm(y).
	\end{align*}
	Therefore, \eqref{eq:master-difference-expanded} is precisely
	\eqref{eq:full-master-difference}.
	
	Finally,
	\[
	\overline U(T,x,m)
	=
	U_1(T,x,m)-U_2(T,x,m)
	=
	G_1(x,m)-G_2(x,m),
	\]
	which proves \eqref{eq:full-master-difference-terminal}.
	The proof is complete.
\end{proof}

Motivated by Proposition~\ref{prop:master-difference-equation}, we
introduce the linear operator
\begin{align}
\mathcal{P}V(t,x,m)
:={}&
-\mathcal{D}_{\beta_1}V(t,x,m)
-\Delta_xV(t,x,m)
+\widehat b(t,x,m)\cdot D_xV(t,x,m)
\nonumber\\
&+
\int_{\mathbb{T}^d}
D_mU_2(t,x,m,y)
\cdot
\widehat A(t,y,m)
D_yV(t,y,m)
\,dm(y).
\label{eq:full-master-operator}
\end{align}
With this notation, equation
\eqref{eq:full-master-difference} can be written as
\[
\mathcal{P}\overline U=0.
\]

\begin{thm}
	\label{thm:full-master-Carleman}
	Let $U_1$ and $U_2$ be as in
	Proposition~\ref{prop:master-difference-equation}, and let
	$Q_T:=(0,T)\times\mathbb{T}^d$.
	Let $m_1=m_1(t,x)$ be an associated flow of probability densities
	generated by $U_1$, with $m_1(0)=m_0$, satisfying
	\begin{equation}\label{eq:U1-measure-flow}
	\partial_t m_1
	-\Delta m_1
	-\operatorname{div}
	\left(
	m_1\beta_1(t,x,m_1(t))
	\right)
	=0
	\qquad \text{in } Q_T.
	\end{equation}
	
	Let $V=V(t,x,m)$ satisfy the regularity assumptions of
	Lemma~\ref{lem:characteristic-derivative}, and assume that
	\[
	v(t,x):=V(t,x,m_1(t))
	\]
	where
	$v\in
	L^2\left(0,T;H^2(\mathbb{T}^d)\right)
	\cap
	H^1\left(0,T;L^2(\mathbb{T}^d)\right).$	
Assume that there exists a constant $M>0$ such that
\begin{equation}\label{eq:master-Carleman-bounds}
\begin{aligned}
&
\sup_{m\in\mathcal{P}(\mathbb{T}^d)}
\left\|
\widehat b(\cdot,\cdot,m)
\right\|_{L^\infty(Q_T)}
+
\sup_{m\in\mathcal{P}(\mathbb{T}^d)}
\left\|
\widehat A(\cdot,\cdot,m)
\right\|_{L^\infty(Q_T)}
\\
&\quad+
\sup_{m\in\mathcal{P}(\mathbb{T}^d)}
\left\|
D_mU_2(\cdot,\cdot,m,\cdot)
\right\|_{L^\infty((0,T)\times\mathbb{T}^d\times\mathbb{T}^d)}
+
\|m_1\|_{L^\infty(Q_T)}
\leq M.
\end{aligned}
\end{equation}
	
	For any $a>0$ and $\nu>1$, define the weight function
	$\varphi_\lambda(t):=e^{\lambda(t+a)^\nu}$.
    Then there exist constants
	$\lambda_0=\lambda_0(T,a,\nu,d,M)>0$
	and
	$C=C(T,a,\nu,d,M)>0$
	such that, for all $\lambda\geq\lambda_0$,
	\begin{equation}
	\begin{aligned}
	&
	\lambda
	\int_{Q_T}
	e^{2\lambda(t+a)^\nu}
	\left|
	D_xV(t,x,m_1(t))
	\right|^2
	\,dxdt
	\\
	&\quad+
	\lambda^2
	\int_{Q_T}
	e^{2\lambda(t+a)^\nu}
	\left|
	V(t,x,m_1(t))
	\right|^2
	\,dxdt
	\\
	&\leq
	C
	\int_{Q_T}
	e^{2\lambda(t+a)^\nu}
	\left|
	\mathcal{P}V(t,x,m_1(t))
	\right|^2
	\,dxdt
	\\
	&\quad+
	Ce^{2\lambda(T+a)^\nu}
	\int_{\mathbb{T}^d}
	\left(
	\left|
	D_xV(T,x,m_1(T))
	\right|^2
	+
	\lambda
	\left|
	V(T,x,m_1(T))
	\right|^2
	\right)
	\,dx .
	\label{eq:full-master-Carleman}
	\end{aligned}
    \end{equation}
	
	In particular, taking $V=\overline U$ and using
	$\mathcal{P}\overline U=0$, we obtain
	\begin{align}
	&
	\lambda
	\int_{Q_T}
	e^{2\lambda(t+a)^\nu}
	\left|
	D_x\overline U(t,x,m_1(t))
	\right|^2
	\,dxdt
	\nonumber\\
	&\quad+
	\lambda^2
	\int_{Q_T}
	e^{2\lambda(t+a)^\nu}
	\left|
	\overline U(t,x,m_1(t))
	\right|^2
	\,dxdt
	\nonumber\\
	&\leq
	Ce^{2\lambda(T+a)^\nu}
	\int_{\mathbb{T}^d}
	\left(
	\left|
	D_x\overline U(T,x,m_1(T))
	\right|^2
	+
	\lambda
	\left|
	\overline U(T,x,m_1(T))
	\right|^2
	\right)
	\,dx .
	\label{eq:master-difference-Carleman}
	\end{align}
\end{thm}

\begin{proof}
	Since $m_1$ satisfies \eqref{eq:U1-measure-flow} with
	$\beta=\beta_1$, Lemma~\ref{lem:characteristic-derivative} yields
	\begin{equation*}\label{eq:Dbeta-to-time}
	\mathcal{D}_{\beta_1}V(t,x,m_1(t))
	=
	\partial_tv(t,x).
	\end{equation*}
	Moreover,
	\[
	D_xV(t,x,m_1(t))=D_xv(t,x),
	\qquad
	\Delta_xV(t,x,m_1(t))=\Delta v(t,x).
	\]
	Therefore, evaluating \eqref{eq:full-master-operator} at
	$m=m_1(t)$, we obtain
	\begin{equation}\label{eq:master-operator-along-flow}
	\mathcal{P}V(t,x,m_1(t))
	=
	-\partial_tv(t,x)
	-\Delta v(t,x)
	+
	\widehat b_1(t,x)\cdot D_xv(t,x)
	+
	\mathcal{N}_1v(t,x),
	\end{equation}
	where
	\[
	\widehat b_1(t,x)
	:=
	\widehat b(t,x,m_1(t))
	\]
	and
	\begin{equation*}\label{eq:N1-definition}
	\mathcal{N}_1v(t,x)
	:=
	\int_{\mathbb{T}^d}
	D_mU_2(t,x,m_1(t),y)
	\cdot
	\widehat A(t,y,m_1(t))
	D_yv(t,y)
	\,dm_1(t,y).
	\end{equation*}
	
	We first estimate the nonlocal term. By the Cauchy--Schwarz inequality
	with respect to the probability measure $m_1(t)$, we have that
	\begin{align*}
	|\mathcal{N}_1v(t,x)|^2
	\leq{}&
	\left[
	\int_{\mathbb{T}^d}
	\left|
	D_mU_2(t,x,m_1(t),y)
	\widehat A(t,y,m_1(t))
	\right|^2
	\,dm_1(t,y)
	\right]
	\nonumber\\
	&\times
	\left[
	\int_{\mathbb{T}^d}
	|D_yv(t,y)|^2
	\,dm_1(t,y)
	\right].
	\label{eq:N1-CS}
	\end{align*}
	By \eqref{eq:master-Carleman-bounds},
	\[
	|\mathcal{N}_1v(t,x)|^2
	\leq
	C_M
	\int_{\mathbb{T}^d}
	|D_yv(t,y)|^2
	\,dm_1(t,y).
	\]
	Since $m_1$ has a bounded density,
	\[
	\int_{\mathbb{T}^d}
	|D_yv(t,y)|^2
	\,dm_1(t,y)
	\leq
	C_M
	\int_{\mathbb{T}^d}
	|D_yv(t,y)|^2\,dy.
	\]
	Consequently,
	\begin{equation}\label{eq:N1-L2-bound}
	\int_{\mathbb{T}^d}
	|\mathcal{N}_1v(t,x)|^2\,dx
	\leq
	C_M
	\int_{\mathbb{T}^d}
	|D_xv(t,x)|^2\,dx.
	\end{equation}
	
We next derive a weighted estimate for the backward parabolic
operator. Set
\[
w(t,x)
=
e^{\lambda(t+a)^\nu}v(t,x).
\]
Then
\begin{equation*}\label{eq:conjugated-master}
e^{\lambda(t+a)^\nu}
\left(
-\partial_tv-\Delta v
\right)
=
-\partial_tw-\Delta w
+
\lambda\nu(t+a)^{\nu-1}w.
\end{equation*}
Hence,
\begin{align}
&
\int_{Q_T}
e^{2\lambda(t+a)^\nu}
|-\partial_tv-\Delta v|^2
\,dxdt
\nonumber\\
&=
\int_{Q_T}
\left|
-\partial_tw-\Delta w
+
\lambda\nu(t+a)^{\nu-1}w
\right|^2
\,dxdt .
\label{eq:master-basic-expansion}
\end{align}
Writing
\[
-\partial_tw-\Delta w+\lambda\nu(t+a)^{\nu-1}w
=
\left(
-\partial_tw+\lambda\nu(t+a)^{\nu-1}w
\right)
-\Delta w,
\]
we obtain
\begin{align}
&
\int_{Q_T}
\left|
-\partial_tw-\Delta w
+
\lambda\nu(t+a)^{\nu-1}w
\right|^2
\,dxdt
\nonumber\\
={}&
\int_{Q_T}
\left|
-\partial_tw+\lambda\nu(t+a)^{\nu-1}w
\right|^2
\,dxdt
+
\int_{Q_T}
|\Delta w|^2
\,dxdt
\nonumber\\
&-
2
\int_{Q_T}
\left(
-\partial_tw+\lambda\nu(t+a)^{\nu-1}w
\right)
\Delta w
\,dxdt .
\label{eq:master-Carleman-split}
\end{align}

For the first term on the right-hand side of
\eqref{eq:master-Carleman-split}, we can derive that
\begin{align*}
&
\int_{Q_T}
\left|
-\partial_tw+\lambda\nu(t+a)^{\nu-1}w
\right|^2
\,dxdt
\nonumber\\
={}&
\int_{Q_T}
|\partial_tw|^2\,dxdt
+
\lambda^2\nu^2
\int_{Q_T}
(t+a)^{2\nu-2}|w|^2\,dxdt
\nonumber\\
&-
2\lambda\nu
\int_{Q_T}
(t+a)^{\nu-1}w\partial_tw
\,dxdt .
\end{align*}
Since
\[
2w\partial_tw=\partial_t(w^2),
\]
integration by parts in time yields
\begin{align*}
&
-2\lambda\nu
\int_{Q_T}
(t+a)^{\nu-1}w\partial_tw
\,dxdt
\nonumber\\
={}&
\lambda\nu(\nu-1)
\int_{Q_T}
(t+a)^{\nu-2}|w|^2\,dxdt
\nonumber\\
&-
\lambda\nu
\left[
(t+a)^{\nu-1}
\int_{\mathbb{T}^d}
|w(t,x)|^2\,dx
\right]_{t=0}^{t=T}.
\end{align*}
Therefore,
\begin{align}
&
\int_{Q_T}
\left|
-\partial_tw+\lambda\nu(t+a)^{\nu-1}w
\right|^2
\,dxdt
\nonumber\\
\geq{}&
\lambda^2\nu^2
\int_{Q_T}
(t+a)^{2\nu-2}|w|^2\,dxdt
\nonumber\\
&-
\lambda\nu
\left[
(t+a)^{\nu-1}
\int_{\mathbb{T}^d}
|w(t,x)|^2\,dx
\right]_{t=0}^{t=T}.
\label{eq:master-Carleman-zero-order}
\end{align}

For the cross term in
\eqref{eq:master-Carleman-split}, integration by parts on
$\mathbb{T}^d$ gives
\begin{align}
&
-2
\int_{Q_T}
\left(
-\partial_tw+\lambda\nu(t+a)^{\nu-1}w
\right)
\Delta w
\,dxdt
\nonumber\\
={}&
2
\int_{Q_T}
\partial_tw\,\Delta w
\,dxdt
-
2\lambda\nu
\int_{Q_T}
(t+a)^{\nu-1}w\Delta w
\,dxdt
\nonumber\\
={}&
-
\left[
\int_{\mathbb{T}^d}
|D_xw(t,x)|^2\,dx
\right]_{t=0}^{t=T}
+
2\lambda\nu
\int_{Q_T}
(t+a)^{\nu-1}|D_xw|^2
\,dxdt .
\label{eq:master-Carleman-gradient}
\end{align}

Combining
\eqref{eq:master-basic-expansion},
\eqref{eq:master-Carleman-split},
\eqref{eq:master-Carleman-zero-order},
and \eqref{eq:master-Carleman-gradient}, we can obtain that
\begin{align}
&
2\lambda\nu
\int_{Q_T}
(t+a)^{\nu-1}|D_xw|^2
\,dxdt
+
\lambda^2\nu^2
\int_{Q_T}
(t+a)^{2\nu-2}|w|^2
\,dxdt
\nonumber\\
&\leq
\int_{Q_T}
e^{2\lambda(t+a)^\nu}
|-\partial_tv-\Delta v|^2
\,dxdt
\nonumber\\
&\quad+
Ce^{2\lambda(T+a)^\nu}
\int_{\mathbb{T}^d}
\left(
|D_xv(T,x)|^2
+
\lambda|v(T,x)|^2
\right)
\,dx .
\label{eq:master-Carleman-preliminary}
\end{align}

Since $a>0$ and $\nu>1$,
\[
(t+a)^{\nu-1}\geq a^{\nu-1},
\qquad
(t+a)^{2\nu-2}\geq a^{2\nu-2},
\qquad t\in[0,T].
\]
Moreover,
\[
w=e^{\lambda(t+a)^\nu}v,
\qquad
D_xw=e^{\lambda(t+a)^\nu}D_xv.
\]
Hence \eqref{eq:master-Carleman-preliminary} implies
\begin{equation}
\begin{aligned}
&
\lambda
\int_{Q_T}
e^{2\lambda(t+a)^\nu}
|D_xv|^2
\,dxdt
+
\lambda^2
\int_{Q_T}
e^{2\lambda(t+a)^\nu}
|v|^2
\,dxdt
\\
&\leq
C
\int_{Q_T}
e^{2\lambda(t+a)^\nu}
|-\partial_tv-\Delta v|^2
\,dxdt
\\
&\quad+
Ce^{2\lambda(T+a)^\nu}
\int_{\mathbb{T}^d}
\left(
|D_xv(T,x)|^2
+
\lambda|v(T,x)|^2
\right)
\,dx .
\label{eq:basic-master-Carleman}
\end{aligned}
\end{equation}

By \eqref{eq:master-operator-along-flow},
\[
-\partial_tv-\Delta v
=
\mathcal{P}V(t,x,m_1(t))
-
\widehat b_1(t,x)\cdot D_xv
-
\mathcal{N}_1v.
\]
Using \eqref{eq:master-Carleman-bounds} and
\eqref{eq:N1-L2-bound}, we obtain
\begin{equation}
\begin{aligned}
&
\int_{Q_T}
e^{2\lambda(t+a)^\nu}
|-\partial_tv-\Delta v|^2
\,dxdt
\\
&\leq
C
\int_{Q_T}
e^{2\lambda(t+a)^\nu}
\left|
\mathcal{P}V(t,x,m_1(t))
\right|^2
\,dxdt
\\
&\quad+
C
\int_{Q_T}
e^{2\lambda(t+a)^\nu}
|D_xv|^2
\,dxdt .
\label{eq:master-lower-order-bound}
\end{aligned}
\end{equation}

Substituting \eqref{eq:master-lower-order-bound} into
\eqref{eq:basic-master-Carleman}, we obtain
\begin{equation}\label{eq:before-master-absorption}
\begin{aligned}
&
\lambda
\int_{Q_T}
e^{2\lambda(t+a)^\nu}|D_xv|^2\,dxdt
+
\lambda^2
\int_{Q_T}
e^{2\lambda(t+a)^\nu}|v|^2\,dxdt
\\
&\leq
C_1
\int_{Q_T}
e^{2\lambda(t+a)^\nu}
\left|
\mathcal{P}V(t,x,m_1(t))
\right|^2
\,dxdt
\\
&\quad+
C_1C_2
\int_{Q_T}
e^{2\lambda(t+a)^\nu}|D_xv|^2\,dxdt
\\
&\quad+
C_1e^{2\lambda(T+a)^\nu}
\int_{\mathbb{T}^d}
\left(
|D_xv(T,x)|^2
+
\lambda|v(T,x)|^2
\right)\,dx,
\end{aligned}
\end{equation}
where $C_1=C_1(T,a,\nu,d)>0$ and $C_2=C_2(M)>0$ are
independent of $\lambda$.

Set
\begin{equation*}
\lambda_0
:=
\max\left\{1,\,2C_1C_2\right\}.
\end{equation*}
Then, for every $\lambda\geq\lambda_0$,
\[
\lambda-C_1C_2\geq\frac{\lambda}{2}.
\]
Hence the second term on the right-hand side of
\eqref{eq:before-master-absorption} can be moved to the left-hand side,
and we have that
\begin{align*}
&
\frac{\lambda}{2}
\int_{Q_T}
e^{2\lambda(t+a)^\nu}|D_xv|^2\,dxdt
+
\lambda^2
\int_{Q_T}
e^{2\lambda(t+a)^\nu}|v|^2\,dxdt
\nonumber\\
&\leq
C_1
\int_{Q_T}
e^{2\lambda(t+a)^\nu}
\left|
\mathcal{P}V(t,x,m_1(t))
\right|^2
\,dxdt
\nonumber\\
&\quad+
C_1e^{2\lambda(T+a)^\nu}
\int_{\mathbb{T}^d}
\left(
|D_xv(T,x)|^2
+
\lambda|v(T,x)|^2
\right)\,dx.
\end{align*}
Therefore, we have that \eqref{eq:full-master-Carleman}.

Finally, taking $V=\overline U$ and using
\[
\mathcal{P}\overline U=0,
\]
we obtain \eqref{eq:master-difference-Carleman}.
This completes the proof.
\end{proof}

\begin{rem}
	The assumptions in \eqref{eq:master-Carleman-bounds} are naturally satisfied for the standard quadratic Hamiltonian
	$H(p)=\frac12 |p|^2$.
	Indeed, in this case $D_pH(p)=p$ and $D_{pp}^2H(p)=I$, so that
	$\widehat A=I,
	\widehat b=\frac12\left(D_xU_1+D_xU_2\right)$.
	Hence, if the spatial gradients of the two classical master fields are
	uniformly bounded, the bounds on $\widehat b$ and $\widehat A$ required
	in \eqref{eq:master-Carleman-bounds} follow immediately. The boundedness assumption on $D_mU$ and the characteristic density $m$ are independent of this particular
	choice of Hamiltonian and are retained as essential a priori
	regularity assumptions.
\end{rem}

\begin{rem}\label{rem_H}
	In fact, the above stability result can also be extended to the case of two master equations with different Hamiltonians. Suppose that $U_i$ corresponds
	to $H_i$, $i=1,2$, and set $\overline{H}:=H_1-H_2$. By linearizing the terms
	involving $H_1$ around $U_2$, the equation for
	$\overline U:=U_1-U_2$ has the same principal and nonlocal structure as
	\eqref{eq:full-master-difference}, with an additional source term
	\[
	R_{\overline{H}}(t,x,m)
	=
	\overline{H}\bigl(x,D_xU_2\bigr)
	+
	\int_{\mathbb T^d}
	D_mU_2(t,x,m,y)\cdot
	D_p\overline{H}\bigl(y,D_yU_2(t,y,m)\bigr)\,dm(y).
	\]
	Thus, perturbations of the Hamiltonian can be incorporated naturally into
	the right-hand side of the Carleman estimate. This gives a master-field
	stability estimate of a similar type to that in Mou, Zhang and Zhou
	\cite{MouZhangZhou2025}, with perturbations in both the Hamiltonian and
	terminal data. The present argument relies on a priori boundedness
	assumptions on the relevant coefficients and solution, without
	requiring a separate uniqueness assumption for the associated
	Hamilton--Jacobi equation.
\end{rem}

\begin{cor}
	Let $U_i$, $i=1,2$, be two classical solutions of the master equations
	corresponding to Hamiltonians $H_i$ and terminal costs $G_i$, respectively,
	with the same running cost $F$. Set
	\[
	\overline U:=U_1-U_2,\qquad
	\overline{H}:=H_1-H_2,\qquad
	\overline{G}:=G_1-G_2.
	\]
	Let $m_1$ be the reference characteristic measure flow generated by $U_1$.
	Assume that the coefficient bounds required in Theorem \ref{thm:full-master-Carleman} hold, and for some $R>0$, 
	\begin{equation}\label{asumption_bounded_DxU}
	\sup_{m\in\mathcal P(\mathbb T^d)}
	\|D_xU_i(\cdot,\cdot,m)\|_{L^\infty(Q_T)}
	\le R, \quad i=1,2.
	\end{equation}
	Define
	\[
	\|\overline{H}\|_{C_p^1(R)}
	:=
	\sup_{\substack{x\in\mathbb T^d, |p|\le R}}
	\left(
	|\overline{H}(x,p)|
	+
	|D_p\overline{H}(x,p)|
	\right).
	\]
	Then there exists a constant $C>0$, depending only on the parameters in
	Theorem \ref{thm:full-master-Carleman} and the corresponding a priori bounds, such that
	\[
	\begin{aligned}
	\|\overline U(\cdot,\cdot,m_1(\cdot))\|_{L^2(0,T;H^1(\mathbb T^d))}^2
    \le
	C\left[
	\|\overline{G}(\cdot,m_1(T))\|_{H^1(\mathbb T^d)}^2
	+
	\|\overline{H}\|_{C_p^1(R)}^2
	\right].
	\end{aligned}
	\]
\end{cor}

\begin{proof}
	As mentioned in the remark \ref{rem_H}, $\overline U$ satisfies
	\[
	\mathcal{P}\overline U=-R_H,
	\]
	where $\mathcal{P}$ is the operator defined as in \eqref{eq:full-master-operator}, with the corresponding
	coefficients constructed from $H_1$. By the boundedness assumption \eqref{eq:master-Carleman-bounds} and \eqref{asumption_bounded_DxU}, we have
	\[
	\|R_H(\cdot,\cdot,m_1(\cdot))\|_{L^2(Q_T)}
	\le C\|\overline H\|_{C_p^1(R)}.
	\]
	Applying Theorem \ref{thm:full-master-Carleman} to $V=\overline U$ and using
	\[
	\overline U(T,x,m_1(T))
	=
	\overline G(x,m_1(T)),
	\]
	yields the desired estimate.
\end{proof}

\section{Lipschitz Stability of the Functional Measure Derivative}\label{section4}

\begin{thm}\label{carleman-estimate-m}
	Let $Q=(0,T)\times \mathbb{T}^{d}\times \mathbb{T}^{d}$,
	and $K\in 
	L^2(0,T;H^2(\mathbb T^d\times\mathbb T^d))
	\cap
	H^1(0,T;L^2(\mathbb T^d\times\mathbb T^d))$.
	For any $a>0$ and $\nu>1$, define the weight function
	$\varphi(t)=e^{\lambda (t+a)^{\nu}}$. Then there exists a constant 
	$C=C(T,a,\nu,d)>0$, such that for all $\lambda\geqslant  1$, we have
	\[
	\begin{aligned}
	&\lambda
	\int_{Q}
	e^{2\lambda (t+a)^{\nu}}
	\Big(
	|D_xK|^{2}
	+
	|D_yK|^{2}
	\Big)\,dx\,dy\,dt
	\\
	&\quad
	+
	\lambda^{2}
	\int_{Q}
	e^{2\lambda (t+a)^{\nu}}
	|K|^{2}\,dx\,dy\,dt
	\\
	&\leq
	C
	\int_{Q}
	e^{2\lambda (t+a)^{\nu}}
	\left|
	-K_t-\Delta_xK-\Delta_yK
	\right|^{2}
	\,dx\,dy\,dt
	\\
	&\quad
	+
	C
	e^{2\lambda (T+a)^{\nu}}
	\int_{\mathbb{T}^{d}\times\mathbb{T}^{d}}
	\Big(
	|D_xK(T,x,y)|^{2}
	+
	|D_yK(T,x,y)|^{2}
	+
	\lambda |K(T,x,y)|^{2}
	\Big)
	\,dx\,dy .
	\end{aligned}
	\]
\end{thm}

\begin{proof}
	Let
	\[
	w=e^{\lambda (t+a)^\nu}K .
	\]
	Then
	\begin{equation*}\label{conjugate}
	e^{\lambda (t+a)^\nu}
	\left(
	-K_t-\Delta_xK-\Delta_yK
	\right)
	=
	-w_t-\Delta_xw-\Delta_yw
	+\lambda\nu(t+a)^{\nu-1}w .
	\end{equation*}
	
	It follows that
	\[
	\begin{aligned}
	&\int_Q
	e^{2\lambda (t+a)^\nu}
	\left|
	-K_t-\Delta_xK-\Delta_yK
	\right|^2
	\\
	&=
	\int_Q
	\left|
	-w_t-\Delta_xw-\Delta_yw
	+\lambda\nu(t+a)^{\nu-1}w
	\right|^2 .
	\end{aligned}
	\]
	
	We write
	\[
	-w_t-\Delta_xw-\Delta_yw
	+\lambda\nu(t+a)^{\nu-1}w
	=
	\left(
	-w_t+\lambda\nu(t+a)^{\nu-1}w
	\right)
	-
	\left(
	\Delta_xw+\Delta_yw
	\right).
	\]
	Hence
	\begin{equation}\label{expand-square}
	\begin{aligned}
	&\int_Q
	\left|
	-w_t-\Delta_xw-\Delta_yw
	+\lambda\nu(t+a)^{\nu-1}w
	\right|^2
	\\
	&=
	\int_Q
	\left|
	-w_t+\lambda\nu(t+a)^{\nu-1}w
	\right|^2+
	\int_Q
	|\Delta_xw+\Delta_yw|^2
	\\
	&\quad
	-
	2\int_Q
	\left(
	-w_t+\lambda\nu(t+a)^{\nu-1}w
	\right)
	\left(
	\Delta_xw+\Delta_yw
	\right).
	\end{aligned}
	\end{equation}
	
	We first estimate the first term on the right-hand side of
	\eqref{expand-square}. Direct expansion gives
	\[
	\begin{aligned}
	&\int_Q
	\left|
	-w_t+\lambda\nu(t+a)^{\nu-1}w
	\right|^2
	\\
	&=
	\int_Q |w_t|^2
	-
	2\lambda\nu
	\int_Q
	(t+a)^{\nu-1}ww_t
	+
	\lambda^2\nu^2
	\int_Q
	(t+a)^{2\nu-2}w^2 .
	\end{aligned}
	\]
	Since $-2ww_t=-\partial_t(w^2)$,
	we have
	\[
	\begin{aligned}
	&-
	2\lambda\nu
	\int_Q
	(t+a)^{\nu-1}ww_t
	\\
	&=
	\lambda\nu(\nu-1)
	\int_Q
	(t+a)^{\nu-2}w^2
	-
	\lambda\nu
	\left[
	(t+a)^{\nu-1}
	\int_{\mathbb T^{2d}} w^2\,dx\,dy
	\right]_0^T .
	\end{aligned}
	\]
	Therefore,
	\begin{equation}\label{time-estimate}
	\begin{aligned}
	&\int_Q
	\left|
	-w_t+\lambda\nu(t+a)^{\nu-1}w
	\right|^2
	\\
	&\ge
	\lambda^2\nu^2
	\int_Q
	(t+a)^{2\nu-2}w^2
	-
	\lambda\nu
	\left[
	(t+a)^{\nu-1}
	\int_{\mathbb T^{2d}} w^2\,dx\,dy
	\right]_0^T .
	\end{aligned}
	\end{equation}
	
	Next we compute the cross term in \eqref{expand-square}. We have
	\[
	\begin{aligned}
	&-2\int_Q
	\left(
	-w_t+\lambda\nu(t+a)^{\nu-1}w
	\right)
	\left(
	\Delta_xw+\Delta_yw
	\right)
	\\
	&=
	2\int_Q w_t(\Delta_xw+\Delta_yw)
	-
	2\lambda\nu
	\int_Q
	(t+a)^{\nu-1}w(\Delta_xw+\Delta_yw).
	\end{aligned}
	\]
	
	For the first term above, integration by parts in $x$ and $y$ gives
	\[
	\begin{aligned}
	2\int_Q w_t(\Delta_xw+\Delta_yw)
	&=
	-2\int_Q
	\left(
	D_xw_t\cdot D_xw
	+
	D_yw_t\cdot D_yw
	\right)
	\\
	&=
	-\int_Q
	\partial_t
	\left(
	|D_xw|^2+|D_yw|^2
	\right)
	\\
	&=
	-
	\left[
	\int_{\mathbb T^{2d}}
	\left(
	|D_xw|^2+|D_yw|^2
	\right)\,dx\,dy
	\right]_0^T .
	\end{aligned}
	\]
	For the second term in \eqref{expand-square}, again using integration by parts in $x$ and $y$, we obtain
	\[
	\begin{aligned}
	&-
	2\lambda\nu
	\int_Q
	(t+a)^{\nu-1}w(\Delta_xw+\Delta_yw)
	\\
	&=
	2\lambda\nu
	\int_Q
	(t+a)^{\nu-1}
	\left(
	|D_xw|^2+|D_yw|^2
	\right).
	\end{aligned}
	\]
	Thus,
	\begin{equation}\label{cross-estimate}
	\begin{aligned}
	&-2\int_Q
	\left(
	-w_t+\lambda\nu(t+a)^{\nu-1}w
	\right)
	\left(
	\Delta_xw+\Delta_yw
	\right)
	\\
	&=
	2\lambda\nu
	\int_Q
	(t+a)^{\nu-1}
	\left(
	|D_xw|^2+|D_yw|^2
	\right)
	-
	\left[
	\int_{\mathbb T^{2d}}
	\left(
	|D_xw|^2+|D_yw|^2
	\right)\,dx\,dy
	\right]_0^T .
	\end{aligned}
	\end{equation}
	
	Combining \eqref{expand-square}, \eqref{time-estimate} and
	\eqref{cross-estimate}, 
	we obtain
	\begin{equation}\label{basic-carleman}
	\begin{aligned}
	&
	2\lambda\nu
	\int_Q
	(t+a)^{\nu-1}
	\left(
	|D_xw|^2+|D_yw|^2
	\right)
	+
	\lambda^2\nu^2
	\int_Q
	(t+a)^{2\nu-2}w^2
	\\
	&\le
	\int_Q
	e^{2\lambda (t+a)^\nu}
	\left|
	-K_t-\Delta_xK-\Delta_yK
	\right|^2
	+
	\lambda\nu
	\left[
	(t+a)^{\nu-1}
	\int_{\mathbb T^{2d}} w^2\,dx\,dy
	\right]_0^T
	\\
	&\quad
	+
	\left[
	\int_{\mathbb T^{2d}}
	\left(
	|D_xw|^2+|D_yw|^2
	\right)\,dx\,dy
	\right]_0^T .
	\end{aligned}
	\end{equation}
	
	Finally, since
	\[
	w=e^{\lambda (t+a)^\nu}K,
	\qquad
	D_xw=e^{\lambda (t+a)^\nu}D_xK,
	\qquad
	D_yw=e^{\lambda (t+a)^\nu}D_yK,
	\]
	and $(t+a)^{\nu-1}$ is bounded from above and below on
	$[0,T]$, \eqref{basic-carleman} implies
	\[
	\begin{aligned}
	&
	\lambda
	\int_Q
	e^{2\lambda (t+a)^\nu}
	\left(
	|D_xK|^2+|D_yK|^2
	\right)
	+
	\lambda^2
	\int_Q
	e^{2\lambda (t+a)^\nu}
	|K|^2
	\\
	&\le
	C
	\int_Q
	e^{2\lambda (t+a)^\nu}
	\left|
	-K_t-\Delta_xK-\Delta_yK
	\right|^2
	\\
	&\quad
	+
	C e^{2\lambda (T+a)^\nu}
	\int_{\mathbb T^{2d}}
	\left(
	|D_xK(T)|^2
	+
	|D_yK(T)|^2
	+
	\lambda |K(T)|^2
	\right)\,dx\,dy .
	\end{aligned}
	\]
	for all $\lambda\geqslant 1$. This completes the proof.
\end{proof}

To derive a stability estimate for the functional derivative of the
master equation, we now compare the measure derivatives associated
with two different measure flows. Let $U_1$ and $U_2$ be the two
master equation solutions introduced in
Proposition~\ref{prop:master-difference-equation}. For $i=1,2$, let
$m_i=m_i(t)$ be the associated measure flow generated by $U_i$ with
the same initial distribution
\[
m_1(0)=m_2(0)=m_0,
\]
and define
\[
u_i(t,x):=U_i(t,x,m_i(t)).
\]
Accordingly, we set
\begin{equation}\label{eq:Ki-bi-Ai}
\begin{aligned}
K_i(t,x,y)
&:=
\frac{\delta U_i}{\delta m}
(t,x,m_i(t),y),
\\
b_i(t,x)
&:=
D_pH\left(x,D_xu_i(t,x)\right),
\\
A_i(t,x)
&:=
D_{pp}^2H\left(x,D_xu_i(t,x)\right),
\qquad i=1,2.
\end{aligned}
\end{equation}
By Theorem~\ref{prop:equation-for-K}, each $K_i$ satisfies
\eqref{eq:K-equation} along its own measure flow $m_i$.

\begin{prop}\label{prop:K-difference-two-flows}
	Let $K_i$, $b_i$, $A_i$, and $m_i$, $i=1,2$, be defined as above.
	Set
	\begin{equation*}\label{eq:bar-quantities}
	\overline K:=K_1-K_2,
	\qquad
	\overline b:=b_1-b_2,
	\qquad
	\overline A:=A_1-A_2,
	\qquad
	\overline m:=m_1-m_2.
	\end{equation*}
	Then $\overline K$ satisfies
	\begin{align}
	-\partial_t\overline K
	-\Delta_x\overline K
	-\Delta_y\overline K
	&+
	b_1(t,x)\cdot D_x\overline K
	+
	b_1(t,y)\cdot D_y\overline K
	\nonumber\\
	&+
	\mathcal N(\overline K)
	=
	R(\overline U,\overline m),
	\label{eq:K-difference-two-flows}
	\end{align}
	where
	\begin{align*}
	\mathcal N(\overline K)
	:={}&
	\int_{\mathbb T^d}
	m_1(t,z)
	D_y\overline K(t,x,z)
	\cdot
	A_1(t,z)D_xK_1(t,z,y)
	\,dz
	\nonumber\\
	&+
	\int_{\mathbb T^d}
	m_1(t,z)
	D_yK_2(t,x,z)
	\cdot
	A_1(t,z)D_x\overline K(t,z,y)
	\,dz,
	\label{eq:N-barK}
	\end{align*}
	and
	\begin{equation}
	\begin{aligned}
	R(\overline U,\overline m)
	:={}&
	\frac{\delta F}{\delta m}
	(x,m_1(t),y)
	-
	\frac{\delta F}{\delta m}
	(x,m_2(t),y)
	\\
	&-
	\overline b(t,x)\cdot D_xK_2(t,x,y)
	-
	\overline b(t,y)\cdot D_yK_2(t,x,y)
	\\
	&-
	\int_{\mathbb T^d}
	\overline m(t,z)
	D_yK_2(t,x,z)
	\cdot
	A_1(t,z)D_xK_2(t,z,y)
	\,dz
	\\
	&-
	\int_{\mathbb T^d}
	m_2(t,z)
	D_yK_2(t,x,z)
	\cdot
	\overline A(t,z)D_xK_2(t,z,y)
	\,dz.
	\label{eq:R-Ubar-mbar}
	\end{aligned}
	\end{equation}
	Moreover,
	\begin{equation}\label{eq:barK-terminal}
	\overline K(T,x,y)
	=
	\frac{\delta G_1}{\delta m}
	(x,m_1(T),y)
	-
	\frac{\delta G_2}{\delta m}
	(x,m_2(T),y).
	\end{equation}
\end{prop}

\begin{proof}
	By Theorem~\ref{prop:equation-for-K}, for $i=1,2$, we have
	\begin{equation*}
	\begin{aligned}
	-\partial_tK_i
	-\Delta_xK_i
	-\Delta_yK_i
	&+
	b_i(t,x)\cdot D_xK_i
	+
	b_i(t,y)\cdot D_yK_i
	\\
	&+
	I_i(t,x,y)
	=
	\frac{\delta F}{\delta m}(x,m_i(t),y),
	\label{eq:Ki-equation-two-flows}
	\end{aligned}
	\end{equation*}
	where
	\begin{equation*}\label{eq:Ii-two-flows}
	I_i(t,x,y)
	:=
	\int_{\mathbb T^d}
	m_i(t,z)
	D_yK_i(t,x,z)
	\cdot
	A_i(t,z)D_xK_i(t,z,y)
	\,dz.
	\end{equation*}
	
	Subtracting the equation for $K_2$ from that for $K_1$, and using
	\[
	K_1=K_2+\overline K,
	\qquad
	b_1=b_2+\overline b,
	\]
	we obtain
	\begin{align}
	-\partial_t\overline K
	-\Delta_x\overline K
	-\Delta_y\overline K
	&+
	b_1(t,x)\cdot D_x\overline K
	+
	b_1(t,y)\cdot D_y\overline K
	\nonumber\\
	&+
	\overline b(t,x)\cdot D_xK_2
	+
	\overline b(t,y)\cdot D_yK_2
	+
	I_1-I_2
	\nonumber\\
	&=
	\frac{\delta F}{\delta m}(x,m_1(t),y)
	-
	\frac{\delta F}{\delta m}(x,m_2(t),y).
	\label{eq:K-difference-before-I}
	\end{align}
	
	Then for the nonlocal terms in \eqref{eq:K-difference-before-I}, using
	\[
	K_1=K_2+\overline K,
	\qquad
	A_1=A_2+\overline A,
	\qquad
	m_1=m_2+\overline m,
	\]
	we have that
	\begin{align}
	I_1-I_2
	={}&
	\int_{\mathbb T^d}
	m_1(t,z)
	D_y\overline K(t,x,z)
	\cdot
	A_1(t,z)D_xK_1(t,z,y)
	\,dz
	\nonumber\\
	&+
	\int_{\mathbb T^d}
	m_1(t,z)
	D_yK_2(t,x,z)
	\cdot
	A_1(t,z)D_x\overline K(t,z,y)
	\,dz
	\nonumber\\
	&+
	\int_{\mathbb T^d}
	\overline m(t,z)
	D_yK_2(t,x,z)
	\cdot
	A_1(t,z)D_xK_2(t,z,y)
	\,dz
	\nonumber\\
	&+
	\int_{\mathbb T^d}
	m_2(t,z)
	D_yK_2(t,x,z)
	\cdot
	\overline A(t,z)D_xK_2(t,z,y)
	\,dz.
	\label{eq:I-difference-two-flows}
	\end{align}
	
    Substituting \eqref{eq:I-difference-two-flows} into
	\eqref{eq:K-difference-before-I}, and keeping the terms containing
	$\overline K$ on the left-hand side, yields
	\eqref{eq:K-difference-two-flows}--\eqref{eq:R-Ubar-mbar}.
	
	Finally, by the terminal condition in
	Theorem~\ref{prop:equation-for-K},
	\[
	K_i(T,x,y)
	=
	\frac{\delta G_i}{\delta m}
	(x,m_i(T),y),
	\qquad i=1,2.
	\]
	Therefore,
	\[
	\overline K(T,x,y)
	=
	\frac{\delta G_1}{\delta m}
	(x,m_1(T),y)
	-
	\frac{\delta G_2}{\delta m}
	(x,m_2(T),y),
	\]
	which proves \eqref{eq:barK-terminal}.
	The proof is complete.
\end{proof}

Motivated by Proposition~\ref{prop:K-difference-two-flows}, we introduce
the following nonlocal parabolic operator. For a function
$V=V(t,x,y)$, define
\begin{align*}
\mathcal{L}_K V
:={}&
-\partial_tV-\Delta_xV-\Delta_yV
+b_1(t,x)\cdot D_xV
+b_1(t,y)\cdot D_yV
\nonumber\\
&+
\int_{\mathbb{T}^d}
m_1(t,z)
D_yV(t,x,z)
\cdot
A_1(t,z)D_xK_1(t,z,y)
\,dz
\nonumber\\
&+
\int_{\mathbb{T}^d}
m_1(t,z)
D_yK_2(t,x,z)
\cdot
A_1(t,z)D_xV(t,z,y)
\,dz.
\label{eq:full-K-operator}
\end{align*}
With this notation, equation
\eqref{eq:K-difference-two-flows} can be written as
\begin{equation}\label{eq:full-K-difference-short}
\mathcal{L}_K\overline K
=
R(\overline U,\overline m).
\end{equation}

\begin{thm}
	\label{thm:full-K-Carleman}
	Let $V=V(t,x,y)$ satisfy the same regularity assumption as in
	Theorem~\ref{carleman-estimate-m}.
	Assume that there exists a constant $M>0$ such that
	\begin{equation}\label{eq:full-K-operator-bounds}
	\begin{aligned}
	\|b_i\|_{L^\infty(Q_T)}
	&+
	\|A_i\|_{L^\infty(Q_T)}
	+
	\|m_i\|_{L^\infty(Q_T)}\\
	&+
	\|D_xK_i\|_{L^\infty(Q)}
	+
	\|D_yK_i\|_{L^\infty(Q)}
	\leq M,   \quad i=1,2.
	\end{aligned}
	\end{equation}
	For any $a>0$ and $\nu>1$, let
	$\varphi_\lambda(t)=e^{\lambda(t+a)^\nu}$
	be the weight function introduced in
	Theorem~\ref{carleman-estimate-m}.
	Then there exist constants
	$\lambda_1=\lambda_1(T,a,\nu,d,M)\geq 1,
	C=C(T,a,\nu,d,M)>0$,
	such that, for all $\lambda\geq\lambda_1$,
	\begin{equation}
	\begin{aligned}
	&
	\lambda
	\int_Q
	e^{2\lambda(t+a)^\nu}
	\left(
	|D_xV|^2+|D_yV|^2
	\right)
	\,dx\,dy\,dt
	+
	\lambda^2
	\int_Q
	e^{2\lambda(t+a)^\nu}
	|V|^2
	\,dx\,dy\,dt
	\\
	&\leq
	C
	\int_Q
	e^{2\lambda(t+a)^\nu}
	|\mathcal{L}_K V|^2
	\,dx\,dy\,dt
	\\
	&\quad+
	Ce^{2\lambda(T+a)^\nu}
	\int_{\mathbb{T}^d\times\mathbb{T}^d}
	\left(
	|D_xV(T,x,y)|^2
	+
	|D_yV(T,x,y)|^2
	+
	\lambda|V(T,x,y)|^2
	\right)
	\,dx\,dy.
	\label{eq:full-K-Carleman}
	\end{aligned}
	\end{equation}
	
	In particular, taking $V=\overline K$ and using
	\eqref{eq:full-K-difference-short}, we obtain
	\begin{equation}
	\begin{aligned}
	&
	\lambda
	\int_Q
	e^{2\lambda(t+a)^\nu}
	\left(
	|D_x\overline K|^2+|D_y\overline K|^2
	\right)
	\,dx\,dy\,dt
	+
	\lambda^2
	\int_Q
	e^{2\lambda(t+a)^\nu}
	|\overline K|^2
	\,dx\,dy\,dt
	\\
	&\leq
	C
	\int_Q
	e^{2\lambda(t+a)^\nu}
	|R(\overline U,\overline m)|^2
	\,dx\,dy\,dt
	\\
	&\quad+
	Ce^{2\lambda(T+a)^\nu}
	\int_{\mathbb{T}^d\times\mathbb{T}^d}
	\left(
	|D_x\overline K(T,x,y)|^2
	+
	|D_y\overline K(T,x,y)|^2
	+
	\lambda|\overline K(T,x,y)|^2
	\right)
	\,dx\,dy.
	\label{eq:barK-Carleman-stability}
	\end{aligned}
	\end{equation}
\end{thm}

\begin{proof}
	Applying Theorem~\ref{carleman-estimate-m} to $V$, we have
	\begin{equation}
	\begin{aligned}
	&
	\lambda
	\int_Q
	e^{2\lambda(t+a)^\nu}
	\left(
	|D_xV|^2+|D_yV|^2
	\right)
	\,dx\,dy\,dt
	+
	\lambda^2
	\int_Q
	e^{2\lambda(t+a)^\nu}
	|V|^2
	\,dx\,dy\,dt
	\\
	&\leq
	C_0
	\int_Q
	e^{2\lambda(t+a)^\nu}
	\left|
	-\partial_tV-\Delta_xV-\Delta_yV
	\right|^2
	\,dx\,dy\,dt
	\\
	&\quad+
	C_0e^{2\lambda(T+a)^\nu}
	\int_{\mathbb{T}^d\times\mathbb{T}^d}
	\left(
	|D_xV(T,x,y)|^2
	+
	|D_yV(T,x,y)|^2
	+
	\lambda|V(T,x,y)|^2
	\right)
	\,dx\,dy,
	\label{eq:apply-basic-K-Carleman}
	\end{aligned}
	\end{equation}
	where
	\[
	C_0=C_0(T,a,\nu,d)>0.
	\]
	
	By the definition of $\mathcal{L}_K$, we have
	\begin{align}
	-\partial_tV-\Delta_xV-\Delta_yV
	={}&
	\mathcal{L}_K V
	-
	b_1(t,x)\cdot D_xV
	-
	b_1(t,y)\cdot D_yV
	\nonumber\\
	&-
	\mathcal{N}_1(V)
	-
	\mathcal{N}_2(V),
	\label{eq:principal-K-from-full}
	\end{align}
	where
	\begin{equation*}\label{eq:N1-V}
	\mathcal{N}_1(V)(t,x,y)
	:=
	\int_{\mathbb{T}^d}
	m_1(t,z)
	D_yV(t,x,z)
	\cdot
	A_1(t,z)D_xK_1(t,z,y)
	\,dz,
	\end{equation*}
	and
	\begin{equation*}\label{eq:N2-V}
	\mathcal{N}_2(V)(t,x,y)
	:=
	\int_{\mathbb{T}^d}
	m_1(t,z)
	D_yK_2(t,x,z)
	\cdot
	A_1(t,z)D_xV(t,z,y)
	\,dz.
	\end{equation*}
	
	By the Cauchy--Schwarz
	inequality,
	\begin{align*}
	|\mathcal{N}_1(V)(t,x,y)|^2
	\leq{}&
	\left(
	\int_{\mathbb{T}^d}
	m_1(t,z)|D_yV(t,x,z)|^2\,dz
	\right)
	\nonumber\\
	&\times
	\left(
	\int_{\mathbb{T}^d}
	m_1(t,z)
	|A_1(t,z)D_xK_1(t,z,y)|^2
	\,dz
	\right).
	\label{eq:N1-CS-new}
	\end{align*}
	Since $m_1(t,\cdot)$ is a probability density, by
	\eqref{eq:full-K-operator-bounds},
	\[
	\int_{\mathbb{T}^d}
	m_1(t,z)
	|A_1(t,z)D_xK_1(t,z,y)|^2
	\,dz
	\leq C_1,
	\]
	where $C_1=C_1(M)>0$. Hence
	\[
	|\mathcal{N}_1(V)(t,x,y)|^2
	\leq
	C_1
	\int_{\mathbb{T}^d}
	m_1(t,z)|D_yV(t,x,z)|^2\,dz.
	\]
	Using again
	\[
	\|m_1\|_{L^\infty(Q_T)}\leq M
	\]
	and integrating with respect to $x$ and $y$, we obtain
	\begin{equation*}\label{eq:N1-L2-new}
	\int_{\mathbb{T}^d\times\mathbb{T}^d}
	|\mathcal{N}_1(V)|^2
	\,dx\,dy
	\leq
	C_1
	\int_{\mathbb{T}^d\times\mathbb{T}^d}
	|D_yV|^2
	\,dx\,dy.
	\end{equation*}
	
	Similarly, by the Cauchy--Schwarz inequality,
	\begin{align*}
	|\mathcal{N}_2(V)(t,x,y)|^2
	\leq{}&
	\left(
	\int_{\mathbb{T}^d}
	m_1(t,z)
	|D_yK_2(t,x,z)A_1(t,z)|^2
	\,dz
	\right)
	\nonumber\\
	&\times
	\left(
	\int_{\mathbb{T}^d}
	m_1(t,z)|D_xV(t,z,y)|^2
	\,dz
	\right).
	\label{eq:N2-CS-new}
	\end{align*}
	Using \eqref{eq:full-K-operator-bounds} and integrating over
	$\mathbb{T}^d\times\mathbb{T}^d$, we obtain
	\begin{equation*}\label{eq:N2-L2-new}
	\int_{\mathbb{T}^d\times\mathbb{T}^d}
	|\mathcal{N}_2(V)|^2
	\,dx\,dy
	\leq
	C_1
	\int_{\mathbb{T}^d\times\mathbb{T}^d}
	|D_xV|^2
	\,dx\,dy.
	\end{equation*}
	
	Therefore,
	\begin{equation}
	\begin{aligned}
	&
	\int_Q
	e^{2\lambda(t+a)^\nu}
	\left(
	|\mathcal{N}_1(V)|^2
	+
	|\mathcal{N}_2(V)|^2
	\right)
	\,dx\,dy\,dt
	\\
	&\leq
	C_1
	\int_Q
	e^{2\lambda(t+a)^\nu}
	\left(
	|D_xV|^2+|D_yV|^2
	\right)
	\,dx\,dy\,dt.
	\label{eq:N-full-bound}
	\end{aligned}
	\end{equation}
	
	Moreover, using
	\[
	\|b_1\|_{L^\infty(Q_T)}\leq M,
	\]
	we have
	\begin{equation}
	\begin{aligned}
	&
	\int_Q
	e^{2\lambda(t+a)^\nu}
	\left|
	b_1(t,x)\cdot D_xV
	+
	b_1(t,y)\cdot D_yV
	\right|^2
	\,dx\,dy\,dt
	\\
	&\leq
	C_1
	\int_Q
	e^{2\lambda(t+a)^\nu}
	\left(
	|D_xV|^2+|D_yV|^2
	\right)
	\,dx\,dy\,dt.
	\label{eq:b-full-bound}
	\end{aligned}
	\end{equation}
	
	Combining
	\eqref{eq:principal-K-from-full},
	\eqref{eq:N-full-bound},
	and \eqref{eq:b-full-bound}, we obtain
    \begin{equation}
	\begin{aligned}
	&
	\int_Q
	e^{2\lambda(t+a)^\nu}
	\left|
	-\partial_tV-\Delta_xV-\Delta_yV
	\right|^2
	\,dx\,dy\,dt
	\\
	&\leq
	C
	\int_Q
	e^{2\lambda(t+a)^\nu}
	|\mathcal{L}_K V|^2
	\,dx\,dy\,dt
	\\
	&\quad+
	C_1
	\int_Q
	e^{2\lambda(t+a)^\nu}
	\left(
	|D_xV|^2+|D_yV|^2
	\right)
	\,dx\,dy\,dt,
	\label{eq:full-K-to-principal}
	\end{aligned}
	\end{equation}
	where $C_1=C_1(M)>0$ is independent of $\lambda$.
	
	Substituting \eqref{eq:full-K-to-principal} into
	\eqref{eq:apply-basic-K-Carleman}, we obtain
	\begin{align}
	&
	\lambda
	\int_Q
	e^{2\lambda(t+a)^\nu}
	\left(
	|D_xV|^2+|D_yV|^2
	\right)
	\,dx\,dy\,dt
	+
	\lambda^2
	\int_Q
	e^{2\lambda(t+a)^\nu}
	|V|^2
	\,dx\,dy\,dt
	\nonumber\\
	&\leq
	C
	\int_Q
	e^{2\lambda(t+a)^\nu}
	|\mathcal{L}_K V|^2
	\,dx\,dy\,dt
	\nonumber\\
	&\quad+
	C_0C_1
	\int_Q
	e^{2\lambda(t+a)^\nu}
	\left(
	|D_xV|^2+|D_yV|^2
	\right)
	\,dx\,dy\,dt
	\nonumber\\
	&\quad+
	Ce^{2\lambda(T+a)^\nu}
	\int_{\mathbb{T}^d\times\mathbb{T}^d}
	\left(
	|D_xV(T,x,y)|^2
	+
	|D_yV(T,x,y)|^2
	+
	\lambda|V(T,x,y)|^2
	\right)
	\,dx\,dy.
	\label{eq:before-K-absorption}
	\end{align}
	
	Set
	\begin{equation*}\label{eq:lambda1-K}
	\lambda_1
	:=
	\max\left\{
	1,\,
	2C_0C_1
	\right\}.
	\end{equation*}
	Then, for every $\lambda\geq\lambda_1$,
	\[
	\lambda-C_0C_1
	\geq
	\frac{\lambda}{2}.
	\]
	Hence the second term on the right-hand side of
	\eqref{eq:before-K-absorption} can be moved to the left-hand side.
	After adjusting the constant $C$, we obtain
	\eqref{eq:full-K-Carleman}.
	
	Finally, taking
	\[
	V=\overline K
	\]
	and using
	\[
	\mathcal{L}_K\overline K
	=
	R(\overline U,\overline m),
	\]
	we obtain \eqref{eq:barK-Carleman-stability}.
	The proof is complete.
\end{proof}

We next estimate the separation between the two measure flows.
Recall that $m_1$ and $m_2$ are generated by $U_1$ and $U_2$,
respectively, from the same initial distribution $m_0$. The following
estimate allows us to control the difference of the two measure flows
by the difference of the master equation solutions evaluated along
the reference flow $m_1$.

\begin{prop}\label{prop:characteristic-separation}
	Let $U_1$ and $U_2$ be as in
	Proposition~\ref{prop:master-difference-equation}, and let
	$m_i$, $u_i$, and $b_i$, $i=1,2$, be defined as in
	Proposition~\ref{prop:K-difference-two-flows}. Assume that
	$m_1(0)=m_2(0)=m_0$.
	Recall that
	$\overline U=U_1-U_2,
	\overline m=m_1-m_2,
	\overline b=b_1-b_2.
	$
	Assume that there exists a constant $M>0$ such that
	\begin{equation}\label{eq:characteristic-separation-assumptions}
	\begin{aligned}
	&
	\|m_i\|_{L^\infty(Q_T)}
	+
	\|\operatorname{div}b_i\|_{L^\infty(Q_T)}
	+
	\|D_{pp}^2H\|_{L^\infty(\mathbb{T}^d\times\mathbb{R}^d)}
	\\
	&\qquad+
	\sup_{m\in\mathcal{P}(\mathbb{T}^d)}
	\left\|
	D_x\frac{\delta U_i}{\delta m}
	(\cdot,\cdot,m,\cdot)
	\right\|_{L^\infty(Q)}
	\leq M, \quad i=1,2.
	\end{aligned}
	\end{equation}
	Then there exists a constant
	$C=C(T,d,M)>0$
	such that
	\begin{align}
	&
	\sup_{t\in[0,T]}
	\|\overline m(t,\cdot)\|_{L^2(\mathbb{T}^d)}^2
	+
	\int_0^T
	\|D_x\overline m(t,\cdot)\|_{L^2(\mathbb{T}^d)}^2\,dt
	\nonumber\\
	&\qquad\leq
	C
	\int_0^T
	\left\|
	D_x\overline U(t,\cdot,m_1(t))
	\right\|_{L^2(\mathbb{T}^d)}^2\,dt.
	\label{eq:characteristic-separation}
	\end{align}
	
	Consequently, for every $\lambda\geq\lambda_0$ for which
	Theorem~\ref{thm:full-master-Carleman} holds,
	\begin{align}
	&
	\sup_{t\in[0,T]}
	\|\overline m(t,\cdot)\|_{L^2(\mathbb{T}^d)}^2
	+
	\int_0^T
	\|D_x\overline m(t,\cdot)\|_{L^2(\mathbb{T}^d)}^2\,dt
	\nonumber\\
	&\qquad\leq
	C
	e^{2\lambda((T+a)^\nu-a^\nu)}
	\left\|
	\overline U(T,\cdot,m_1(T))
	\right\|_{H^1(\mathbb{T}^d)}^2 .
	\label{eq:characteristic-separation-terminal}
	\end{align}
\end{prop}

\begin{proof}
	For $i=1,2$, the measure flow $m_i$ satisfies
	\[
	\partial_tm_i-\Delta m_i-\operatorname{div}(m_i b_i)=0.
	\]
	Subtracting the equation for $m_2$ from that for $m_1$, and using
	\[
	m_1b_1-m_2b_2
	=
	\overline m\,b_1+m_2\overline b,
	\]
	we obtain that
	\begin{equation}\label{eq:bar-m-equation}
	\partial_t\overline m
	-\Delta\overline m
	-\operatorname{div}(\overline m\,b_1)
	=
	\operatorname{div}(m_2\overline b),
	\end{equation}
	with
	$\overline m(0,x)=0$.

	We first estimate $\overline b$. By the definition in \eqref{eq:Ki-bi-Ai}, we have that
	\begin{align}
	\overline b(t,x)
	={}&
	D_pH\left(
	x,D_xU_1(t,x,m_1(t))
	\right)
	-
	D_pH\left(
	x,D_xU_2(t,x,m_1(t))
	\right)
	\nonumber\\
	&+
	D_pH\left(
	x,D_xU_2(t,x,m_1(t))
	\right)
	-
	D_pH\left(
	x,D_xU_2(t,x,m_2(t))
	\right).
	\label{eq:b-difference-decomposition}
	\end{align}
	
	By \eqref{eq:characteristic-separation-assumptions},
	the first difference on the right-hand side satisfies
	\begin{align}
	&
	\left|
	D_pH\left(
	x,D_xU_1(t,x,m_1(t))
	\right)
	-
	D_pH\left(
	x,D_xU_2(t,x,m_1(t))
	\right)
	\right|
	\nonumber\\
	&\qquad\leq
	C
	\left|
	D_x\overline U(t,x,m_1(t))
	\right|.
	\label{eq:b-difference-first-part}
	\end{align}
	
	To estimate the second difference, for $s\in[0,1]$ define
	\[
	m_s(t):=(1-s)m_2(t)+sm_1(t).
	\]
	Then we can obtain that
	\begin{equation}
	\begin{aligned}
	&
	D_xU_2(t,x,m_1(t))
	-
	D_xU_2(t,x,m_2(t))
	\nonumber\\
	&\qquad=
	\int_0^1
	\int_{\mathbb{T}^d}
	D_x\frac{\delta U_2}{\delta m}
	(t,x,m_s(t),y)
	\,\overline m(t,y)
	\,dy\,ds.
	\label{eq:U2-measure-difference}
	\end{aligned}
	\end{equation}
	Hence, by
	\eqref{eq:characteristic-separation-assumptions},
	\begin{equation}\label{estimate_DU}
	\begin{aligned}
	&
	\left|
	D_xU_2(t,x,m_1(t))
	-
	D_xU_2(t,x,m_2(t))
	\right|
	\\
	&\qquad\leq
	C
	\|\overline m(t,\cdot)\|_{L^1(\mathbb{T}^d)}
	\leq
	C
	\|\overline m(t,\cdot)\|_{L^2(\mathbb{T}^d)}.
	\end{aligned}
	\end{equation}
	
	Using again the boundedness of $D_{pp}^2H$, we obtain from
	\eqref{eq:b-difference-decomposition}--%
	\eqref{eq:b-difference-first-part}
	\begin{equation}\label{eq:b-difference-L2}
	\|\overline b(t,\cdot)\|_{L^2(\mathbb{T}^d)}^2
	\leq
	C
	\left(
	\left\|
	D_x\overline U(t,\cdot,m_1(t))
	\right\|_{L^2(\mathbb{T}^d)}^2
	+
	\|\overline m(t,\cdot)\|_{L^2(\mathbb{T}^d)}^2
	\right).
	\end{equation}
	
	We now derive an energy estimate for
	\eqref{eq:bar-m-equation}. Multiplying
	\eqref{eq:bar-m-equation} by $\overline m$ and integrating over
	$\mathbb{T}^d$, we obtain
	\begin{align*}
	&
	\frac{1}{2}
	\frac{d}{dt}
	\|\overline m(t,\cdot)\|_{L^2(\mathbb{T}^d)}^2
	+
	\|D_x\overline m(t,\cdot)\|_{L^2(\mathbb{T}^d)}^2
	\\
	&\quad
	-
	\frac{1}{2}
	\int_{\mathbb{T}^d}
	\operatorname{div}b_1(t,x)
	|\overline m(t,x)|^2\,dx
	\\
	&=
	-
	\int_{\mathbb{T}^d}
	m_2(t,x)\overline b(t,x)
	\cdot D_x\overline m(t,x)\,dx.
	\label{eq:bar-m-energy}
	\end{align*}
	
	By \eqref{eq:characteristic-separation-assumptions}
	and Young's inequality,
	\begin{equation*}
	\begin{aligned}
	\left|
	\int_{\mathbb{T}^d}
	m_2\overline b\cdot D_x\overline m\,dx
	\right|
	\leq
	\frac{1}{2}
	\|D_x\overline m\|_{L^2}^2
	+
	C\|\overline b\|_{L^2}^2,
	\label{eq:bar-m-source-estimate}
	\end{aligned}
	\end{equation*}
	while
	\[
	\left|
	\int_{\mathbb{T}^d}
	\operatorname{div}b_1
	|\overline m|^2\,dx
	\right|
	\leq
	C\|\overline m\|_{L^2}^2.
	\]
	Consequently, using \eqref{eq:b-difference-L2}, we obtain
	\begin{align}
	&
	\frac{d}{dt}
	\|\overline m(t,\cdot)\|_{L^2}^2
	+
	\|D_x\overline m(t,\cdot)\|_{L^2}^2
	\nonumber\\
	&\qquad\leq
	C\|\overline m(t,\cdot)\|_{L^2}^2
	+
	C
	\left\|
	D_x\overline U(t,\cdot,m_1(t))
	\right\|_{L^2}^2.
	\label{eq:bar-m-differential-inequality}
	\end{align}
	
	Since $\overline m(0)=0$, Gronwall's inequality gives
	\begin{equation}\label{eq:bar-m-L2-control}
	\sup_{t\in[0,T]}
	\|\overline m(t,\cdot)\|_{L^2}^2
	\leq
	C
	\int_0^T
	\left\|
	D_x\overline U(t,\cdot,m_1(t))
	\right\|_{L^2}^2\,dt.
	\end{equation}
	Integrating \eqref{eq:bar-m-differential-inequality} over $(0,T)$
	and using \eqref{eq:bar-m-L2-control}, we also obtain
	\[
	\int_0^T
	\|D_x\overline m(t,\cdot)\|_{L^2}^2\,dt
	\leq
	C
	\int_0^T
	\left\|
	D_x\overline U(t,\cdot,m_1(t))
	\right\|_{L^2}^2\,dt.
	\]
	This proves \eqref{eq:characteristic-separation}.
	
	Finally, applying
	Theorem~\ref{thm:full-master-Carleman} to
	$V=\overline U$ and using
	\[
	\mathcal{P}\overline U=0,
	\]
	we obtain
	\begin{align*}
	&
	\lambda
	\int_0^T
	e^{2\lambda(t+a)^\nu}
	\left\|
	D_x\overline U(t,\cdot,m_1(t))
	\right\|_{L^2}^2\,dt
	\nonumber\\
	&\qquad\leq
	Ce^{2\lambda(T+a)^\nu}
	\left(
	\left\|
	D_x\overline U(T,\cdot,m_1(T))
	\right\|_{L^2}^2
	+
	\lambda
	\left\|
	\overline U(T,\cdot,m_1(T))
	\right\|_{L^2}^2
	\right).
	\label{eq:U-gradient-terminal}
	\end{align*}
	Since
	\[
	e^{2\lambda(t+a)^\nu}
	\geq e^{2\lambda a^\nu},
	\qquad t\in[0,T],
	\]
	it follows that
	\begin{equation}
	\begin{aligned}
	&
	\int_0^T
	\left\|
	D_x\overline U(t,\cdot,m_1(t))
	\right\|_{L^2}^2\,dt
	\nonumber\\
	&\qquad\leq
	C
	e^{2\lambda((T+a)^\nu-a^\nu)}
	\left(
	\frac{1}{\lambda}
	\left\|
	D_x\overline U(T,\cdot,m_1(T))
	\right\|_{L^2}^2
	+
	\left\|
	\overline U(T,\cdot,m_1(T))
	\right\|_{L^2}^2
	\right)
	\nonumber\\
	&\qquad\leq
	C
	e^{2\lambda((T+a)^\nu-a^\nu)}
	\left\|
	\overline U(T,\cdot,m_1(T))
	\right\|_{H^1(\mathbb{T}^d)}^2,
	\label{eq:U-gradient-terminal-unweighted}
	\end{aligned}
	\end{equation}
	where we used $\lambda\geq\lambda_0\geq1$ in the last inequality.
	Combining this estimate with
	\eqref{eq:characteristic-separation} yields
	\eqref{eq:characteristic-separation-terminal}.
	The proof is complete.
\end{proof}

We now combine the estimates obtained in the previous sections to
derive a stability estimate for the master equation and its functional
derivative along the reference measure flow $m_1$.

Let $U_1$ and $U_2$ be the two master equation solutions introduced in
Proposition~\ref{prop:master-difference-equation}. Recall that
\[
\overline U:=U_1-U_2.
\]
Let $m_i$, $K_i$, $b_i$, and $A_i$, $i=1,2$, be defined as in
Proposition~\ref{prop:K-difference-two-flows}, with
\[
m_1(0)=m_2(0)=m_0.
\]
We further recall that
\[
\overline m:=m_1-m_2,
\qquad
\overline K:=K_1-K_2,
\qquad
\overline b:=b_1-b_2,
\qquad
\overline A:=A_1-A_2.
\]

\begin{thm}
	\label{thm:master-stability}
	Assume that the hypotheses of
	Theorem~\ref{thm:full-master-Carleman} are satisfied for
	$V=\overline U$, that the hypotheses of
	Theorem~\ref{thm:full-K-Carleman} are satisfied for
	$V=\overline K$, and that the hypotheses of
	Proposition~\ref{prop:characteristic-separation} hold.	
	Assume, in addition, that there exists a constant $M>0$ such that
	\begin{equation}\label{eq:Hessian-H-Lipschitz-final}
	\left|
	D_{pp}^2H(x,p_1)-D_{pp}^2H(x,p_2)
	\right|
	\leq
	M|p_1-p_2|,
	\qquad
	x\in\mathbb{T}^d,\quad p_1,p_2\in\mathbb{R}^d,
	\end{equation}
	and
	\begin{equation}\label{eq:Fm-measure-Lipschitz-final}
	\left\|
	\frac{\delta F}{\delta m}(\cdot,m,\cdot)
	-
	\frac{\delta F}{\delta m}(\cdot,m',\cdot)
	\right\|_{L^2(\mathbb{T}^d\times\mathbb{T}^d)}
	\leq
	M\|m-m'\|_{L^2(\mathbb{T}^d)}
	\end{equation}
	for all admissible probability densities $m, m'\in\mathcal P_M(\mathbb T^d)$.
	Assume also that the functional derivative of $U_2$ is Lipschitz
	continuous with respect to the measure variable in the sense that
	\begin{equation}\label{eq:Um-measure-Lipschitz-final}
	\left\|
	\frac{\delta U_2}{\delta m}(t,\cdot,m,\cdot)
	-
	\frac{\delta U_2}{\delta m}(t,\cdot,m',\cdot)
	\right\|_{H^1(\mathbb{T}^d\times\mathbb{T}^d)}
	\leq
	M\|m-m'\|_{L^2(\mathbb{T}^d)},
	\end{equation}
	for all $t\in[0,T]$ and all admissible probability densities
	$m, m'\in\mathcal P_M(\mathbb T^d)$.

	Then there exists a constant
	$C=C(T,a,\nu,d,M)>0$
	such that
	\begin{align}
	&
	\left\|
	\overline U(\cdot,\cdot,m_1(\cdot))
	\right\|_{L^2(0,T;H^1(\mathbb{T}^d))}^2
	+
	\left\|
	\frac{\delta\overline U}{\delta m}
	(\cdot,\cdot,m_1(\cdot),\cdot)
	\right\|_{L^2(0,T;H^1(\mathbb{T}^d\times\mathbb{T}^d))}^2
	\nonumber\\
	&\leq
	C
	\left[
	\left\|
	\overline U(T,\cdot,m_1(T))
	\right\|_{H^1(\mathbb{T}^d)}^2
	+
	\left\|
	\frac{\delta\overline U}{\delta m}
	(T,\cdot,m_1(T),\cdot)
	\right\|_{H^1(\mathbb{T}^d\times\mathbb{T}^d)}^2
	\right].
	\label{eq:final-master-stability}
	\end{align}
\end{thm}

\begin{proof}
	Let $\lambda_0$ and $\lambda_1$ be the Carleman thresholds given by
	Theorem~\ref{thm:full-master-Carleman} and
	Theorem~\ref{thm:full-K-Carleman}, respectively, and set
	\begin{equation*}\label{eq:lambda-star}
	\lambda_*:=\max\{\lambda_0,\lambda_1\}.
	\end{equation*}
	
	Applying Theorem~\ref{thm:full-master-Carleman} to
	$V=\overline U$ and using
	\[
	\mathcal P\overline U=0,
	\]
	we obtain
	\begin{equation*}
	\begin{aligned}
	&
	\lambda_*
	\int_0^T
	e^{2\lambda_*(t+a)^\nu}
	\left\|
	D_x\overline U(t,\cdot,m_1(t))
	\right\|_{L^2(\mathbb{T}^d)}^2\,dt
	\nonumber\\
	&\quad+
	\lambda_*^2
	\int_0^T
	e^{2\lambda_*(t+a)^\nu}
	\left\|
	\overline U(t,\cdot,m_1(t))
	\right\|_{L^2(\mathbb{T}^d)}^2\,dt
	\nonumber\\
	&\leq
	Ce^{2\lambda_*(T+a)^\nu}
	\left\|
	\overline U(T,\cdot,m_1(T))
	\right\|_{H^1(\mathbb{T}^d)}^2.
	\label{eq:U-final-step}
	\end{aligned}
	\end{equation*}
	Since $\lambda_*$ is fixed and
	\[
	e^{2\lambda_*(t+a)^\nu}
	\geq
	e^{2\lambda_*a^\nu},
	\qquad t\in[0,T],
	\]
	it follows that
	\begin{equation}\label{eq:U-unweighted-final}
	\left\|
	\overline U(\cdot,\cdot,m_1(\cdot))
	\right\|_{L^2(0,T;H^1(\mathbb{T}^d))}^2
	\leq
	C
	\left\|
	\overline U(T,\cdot,m_1(T))
	\right\|_{H^1(\mathbb{T}^d)}^2.
	\end{equation}
	
	By Proposition~\ref{prop:characteristic-separation},
	\begin{equation*}
	\begin{aligned}
	&
	\sup_{t\in[0,T]}
	\|\overline m(t,\cdot)\|_{L^2(\mathbb{T}^d)}^2
	+
	\int_0^T
	\|D_x\overline m(t,\cdot)\|_{L^2(\mathbb{T}^d)}^2\,dt
	\nonumber\\
	&\leq
	C
	\int_0^T
	\left\|
	D_x\overline U(t,\cdot,m_1(t))
	\right\|_{L^2(\mathbb{T}^d)}^2\,dt.
	\label{eq:mbar-control-final}
	\end{aligned}
	\end{equation*}
	Hence, by \eqref{eq:U-unweighted-final},
	\begin{equation}\label{eq:mbar-final}
	\begin{aligned}
	&
	\sup_{t\in[0,T]}
	\|\overline m(t,\cdot)\|_{L^2(\mathbb{T}^d)}^2
	+
	\int_0^T
	\|D_x\overline m(t,\cdot)\|_{L^2(\mathbb{T}^d)}^2\,dt
	\\
	&\leq
	C
	\left\|
	\overline U(T,\cdot,m_1(T))
	\right\|_{H^1(\mathbb{T}^d)}^2.
	\end{aligned}
	\end{equation}
	
	Recall from Proposition~\ref{prop:K-difference-two-flows} that
    \begin{equation*}
	\begin{aligned}
	R(\overline U,\overline m)
	={}&
	\frac{\delta F}{\delta m}(x,m_1(t),y)
	-
	\frac{\delta F}{\delta m}(x,m_2(t),y)
	\\
	&-
	\overline b(t,x)\cdot D_xK_2(t,x,y)
	-
	\overline b(t,y)\cdot D_yK_2(t,x,y)
	\\
	&-
	\int_{\mathbb{T}^d}
	\overline m(t,z)
	D_yK_2(t,x,z)
	\cdot
	A_1(t,z)D_xK_2(t,z,y)\,dz
	\\
	&-
	\int_{\mathbb{T}^d}
	m_2(t,z)
	D_yK_2(t,x,z)
	\cdot
	\overline A(t,z)D_xK_2(t,z,y)\,dz.
	\label{eq:R-final-recall}
	\end{aligned}
	\end{equation*}
	
	By Proposition~\ref{prop:characteristic-separation},
	\begin{equation}\label{eq:bbar-final}
	\|\overline b(t,\cdot)\|_{L^2(\mathbb{T}^d)}^2
	\leq
	C
	\left(
	\left\|
	D_x\overline U(t,\cdot,m_1(t))
	\right\|_{L^2(\mathbb{T}^d)}^2
	+
	\|\overline m(t,\cdot)\|_{L^2(\mathbb{T}^d)}^2
	\right).
	\end{equation}
	
	For $\overline A=A_1-A_2$, we have that
    \begin{equation*}
	\begin{aligned}
	\overline A(t,x)
	={}&
	D_{pp}^2H
	\left(
	x,D_xU_1(t,x,m_1(t))
	\right)
	-
	D_{pp}^2H
	\left(
	x,D_xU_2(t,x,m_1(t))
	\right)
	\\
	&+
	D_{pp}^2H
	\left(
	x,D_xU_2(t,x,m_1(t))
	\right)
	-
	D_{pp}^2H
	\left(
	x,D_xU_2(t,x,m_2(t))
	\right).
	\label{eq:Abar-decomposition}
	\end{aligned}
	\end{equation*}
	Using \eqref{eq:Hessian-H-Lipschitz-final} together with the estimate \eqref{estimate_DU}
	obtained in Proposition~\ref{prop:characteristic-separation}, we get
	\begin{equation}\label{eq:Abar-final}
	\|\overline A(t,\cdot)\|_{L^2(\mathbb{T}^d)}^2
	\leq
	C
	\left(
	\left\|
	D_x\overline U(t,\cdot,m_1(t))
	\right\|_{L^2(\mathbb{T}^d)}^2
	+
	\|\overline m(t,\cdot)\|_{L^2(\mathbb{T}^d)}^2
	\right).
	\end{equation}
	
	By \eqref{eq:Fm-measure-Lipschitz-final},
	\begin{equation*}\label{eq:Fm-final-bound}
	\left\|
	\frac{\delta F}{\delta m}(\cdot,m_1(t),\cdot)
	-
	\frac{\delta F}{\delta m}(\cdot,m_2(t),\cdot)
	\right\|_{L^2(\mathbb{T}^d\times\mathbb{T}^d)}
	\leq
	C
	\|\overline m(t,\cdot)\|_{L^2(\mathbb{T}^d)}.
	\end{equation*}
	
	Using the bounds inherited from
	Theorem~\ref{thm:full-K-Carleman} and
	Proposition~\ref{prop:characteristic-separation}, together with
	\eqref{eq:Abar-final}, the Cauchy--Schwarz inequality gives
	\begin{equation*}
	\begin{aligned}
	&
	\left\|
	\int_{\mathbb{T}^d}
	\overline m(t,z)
	D_yK_2(t,\cdot,z)
	\cdot
	A_1(t,z)D_xK_2(t,z,\cdot)\,dz
	\right\|_{L^2(\mathbb{T}^d\times\mathbb{T}^d)}
	\nonumber\\
	&\qquad\leq
	C
	\|\overline m(t,\cdot)\|_{L^2(\mathbb{T}^d)},
	\label{eq:R-m-integral-final}
	\end{aligned}
	\end{equation*}
	and
	\begin{align}
	&
	\left\|
	\int_{\mathbb{T}^d}
	m_2(t,z)
	D_yK_2(t,\cdot,z)
	\cdot
	\overline A(t,z)D_xK_2(t,z,\cdot)\,dz
	\right\|_{L^2(\mathbb{T}^d\times\mathbb{T}^d)}
	\nonumber\\
	&\qquad\leq
	C
	\|\overline A(t,\cdot)\|_{L^2(\mathbb{T}^d)}.
	\label{eq:R-A-integral-final}
	\end{align}
	
	Combining
	\eqref{eq:bbar-final}--\eqref{eq:R-A-integral-final}, we obtain
	\begin{align*}
	\|R(\overline U,\overline m)(t)\|_{L^2(\mathbb{T}^d\times\mathbb{T}^d)}^2
	\leq
	C
	\left(
	\left\|
	D_x\overline U(t,\cdot,m_1(t))
	\right\|_{L^2(\mathbb{T}^d)}^2
	+
	\|\overline m(t,\cdot)\|_{L^2(\mathbb{T}^d)}^2
	\right).
	\label{eq:R-final-estimate}
	\end{align*}
	Consequently, by \eqref{eq:U-unweighted-final} and
	\eqref{eq:mbar-final},
	\begin{equation}\label{eq:R-final-terminal-control}
	\int_0^T
	\|R(\overline U,\overline m)(t)\|_{L^2(\mathbb{T}^d\times\mathbb{T}^d)}^2
	\,dt
	\leq
	C
	\left\|
	\overline U(T,\cdot,m_1(T))
	\right\|_{H^1(\mathbb{T}^d)}^2.
	\end{equation}
	
	Applying Theorem~\ref{thm:full-K-Carleman} to
	$V=\overline K$ with $\lambda=\lambda_*$, and using
	\[
	\mathcal L_K\overline K
	=
	R(\overline U,\overline m),
	\]
	together with \eqref{eq:R-final-terminal-control}, we obtain
	\begin{align}
	\|\overline K\|_{L^2(0,T;H^1(\mathbb{T}^d\times\mathbb{T}^d))}^2
	\leq
	C
	\Big[
	&
	\left\|
	\overline U(T,\cdot,m_1(T))
	\right\|_{H^1(\mathbb{T}^d)}^2
	\nonumber\\
	&+
	\|\overline K(T,\cdot,\cdot)\|_{H^1(\mathbb{T}^d\times\mathbb{T}^d)}^2
	\Big].
	\label{eq:barK-unweighted}
	\end{align}
	
	We next estimate the terminal value of $\overline K$. By definition,
	\begin{equation*}
	\begin{aligned}
	\overline K(T,x,y)
	={}&
	\frac{\delta U_1}{\delta m}
	(T,x,m_1(T),y)
	-
	\frac{\delta U_2}{\delta m}
	(T,x,m_2(T),y)
	\nonumber\\
	={}&
	\frac{\delta\overline U}{\delta m}
	(T,x,m_1(T),y)
	\nonumber\\
	&+
	\frac{\delta U_2}{\delta m}
	(T,x,m_1(T),y)
	-
	\frac{\delta U_2}{\delta m}
	(T,x,m_2(T),y).
	\label{eq:barK-terminal-decomposition-final}
	\end{aligned}
	\end{equation*}
	Using \eqref{eq:Um-measure-Lipschitz-final} and
	\eqref{eq:mbar-final}, we obtain
	\begin{align}
	\|\overline K(T,\cdot,\cdot)\|_{H^1(\mathbb{T}^d\times\mathbb{T}^d)}^2
	\leq
	C
	\Big[
	&
	\left\|
	\frac{\delta\overline U}{\delta m}
	(T,\cdot,m_1(T),\cdot)
	\right\|_{H^1(\mathbb{T}^d\times\mathbb{T}^d)}^2
	\nonumber\\
	&+
	\left\|
	\overline U(T,\cdot,m_1(T))
	\right\|_{H^1(\mathbb{T}^d)}^2
	\Big].
	\label{eq:barK-terminal-final}
	\end{align}
	
	Combining \eqref{eq:barK-unweighted} and
	\eqref{eq:barK-terminal-final}, we obtain
	\begin{equation}
	\begin{aligned}
	\|\overline K\|_{L^2(0,T;H^1(\mathbb{T}^d\times\mathbb{T}^d))}^2
	\leq
	C
	\Big[
	&
	\left\|
	\overline U(T,\cdot,m_1(T))
	\right\|_{H^1(\mathbb{T}^d)}^2
	\\
	&+
	\left\|
	\frac{\delta\overline U}{\delta m}
	(T,\cdot,m_1(T),\cdot)
	\right\|_{H^1(\mathbb{T}^d\times\mathbb{T}^d)}^2
	\Big].
	\label{eq:barK-final-control}
	\end{aligned}
	\end{equation}

	Besides,
	\begin{equation}
	\begin{aligned}
	\frac{\delta\overline U}{\delta m}
	(t,x,m_1(t),y)
	={}&
	\overline K(t,x,y)
	\nonumber\\
	&+
	\frac{\delta U_2}{\delta m}
	(t,x,m_2(t),y)
	-
	\frac{\delta U_2}{\delta m}
	(t,x,m_1(t),y).
	\label{eq:K-transfer-final}
	\end{aligned}
	\end{equation}
	Hence, by \eqref{eq:Um-measure-Lipschitz-final},
	\begin{equation*}
	\begin{aligned}
	&
	\left\|
	\frac{\delta\overline U}{\delta m}
	(t,\cdot,m_1(t),\cdot)
	\right\|_{H^1(\mathbb{T}^d\times\mathbb{T}^d)}^2
	\nonumber\\
	&\qquad\leq
	C
	\left(
	\|\overline K(t,\cdot,\cdot)\|_{H^1(\mathbb{T}^d\times\mathbb{T}^d)}^2
	+
	\|\overline m(t,\cdot)\|_{L^2(\mathbb{T}^d)}^2
	\right).
	\label{eq:K-transfer-estimate-final}
	\end{aligned}
	\end{equation*}
	Integrating over $(0,T)$ and using
	\eqref{eq:mbar-final} and \eqref{eq:barK-final-control}, we obtain
	\begin{align}
	&
	\left\|
	\frac{\delta\overline U}{\delta m}
	(\cdot,\cdot,m_1(\cdot),\cdot)
	\right\|_{L^2(0,T;H^1(\mathbb{T}^d\times\mathbb{T}^d))}^2
	\nonumber\\
	&\leq
	C
	\left[
	\left\|
	\overline U(T,\cdot,m_1(T))
	\right\|_{H^1(\mathbb{T}^d)}^2
	+
	\left\|
	\frac{\delta\overline U}{\delta m}
	(T,\cdot,m_1(T),\cdot)
	\right\|_{H^1(\mathbb{T}^d\times\mathbb{T}^d)}^2
	\right].
	\label{eq:Km-final-same-flow}
	\end{align}
	
	Combining \eqref{eq:U-unweighted-final} and
	\eqref{eq:Km-final-same-flow} yields
	\eqref{eq:final-master-stability}.
	This completes the proof.
\end{proof}

\begin{rem}
	Although the estimate in Theorem \ref{thm:full-master-Carleman} and Theorem~\ref{thm:master-stability} are written
	along the reference flow $m_1$, this flow is not fixed a priori.
	Indeed, the initial distribution $m_0$ can be chosen arbitrarily within
	the admissible class, and the same argument applies to the corresponding
	characteristic flow generated by $U_1$, provided that the a priori
	bounds appearing in the assumptions are satisfied uniformly.
	Therefore, the stability estimate should be understood as a
	flow-wise estimate valid for an arbitrary admissible characteristic
	measure flow.
\end{rem}

\begin{rem}
	The stability results above also yield corresponding uniqueness consequences.
	If $G_1=G_2$, the master-field stability estimate in Theorem \ref{thm:full-master-Carleman} implies that
	$U_1=U_2$ along every admissible characteristic measure flow. Under the
	additional assumptions of Theorem \ref{thm:master-stability}, the same terminal data further imply
	the coincidence of their functional measure derivatives along these flows.
	Moreover, since $\delta U/\delta m$ represents the first-order regularity of
	the master field with respect to the measure variable, Theorem \ref{thm:master-stability} may also be
	viewed as a stability result at a higher regularity level, controlling not only
	the master field itself but also its first functional derivative with respect
	to the measure variable.
\end{rem}

\end{document}